\documentclass[11pt,oneside,english]{amsart}
\usepackage[T1]{fontenc}
\usepackage[latin9]{inputenc}
\usepackage[active]{srcltx}
\usepackage{amsthm}
\usepackage{amstext}
\usepackage{amssymb}
\usepackage{esint}
\usepackage{hyperref}

\makeatletter
\theoremstyle{plain}
\newtheorem{thm}{\protect\theoremname}[section]

\theoremstyle{plain}
\newtheorem{lem}[thm]{\protect\lemmaname}

\theoremstyle{remark}
\newtheorem{rem}[thm]{\protect\remarkname}

\usepackage{array}\usepackage{float}\usepackage{graphics}\usepackage{enumitem}\usepackage{fancyhdr}\usepackage{amsthm}\usepackage{amsbsy}%, showkeys}
\usepackage[normalem]{ulem}

\numberwithin{equation}{section}
\numberwithin{figure}{section}

\theoremstyle{plain}

 \newtheorem{corollary}[thm]{Corollary}

\theoremstyle{definition}

\theoremstyle{remark}
\newcommand{\mfp}{\mathfrak{p}}

\newcommand{\C}{\ensuremath{\mathbb{C}}}
\newcommand{\F}{\ensuremath{\mathbb{F}}}

\newcommand{\Fq}{\ensuremath{\mathbb{F}_q}}
\newcommand{\N}{\ensuremath{\mathbb{N}}}
\newcommand{\Q}{\ensuremath{\mathbb{Q}}}
\newcommand{\R}{\ensuremath{\mathbb{R}}}

\newcommand{\T}{\ensuremath{\mathcal{T}}}

\newcommand{\Z}{\ensuremath{\mathbb{Z}}}

\newcommand{\bfa}{\ensuremath{\mathbf{a}}}

\newcommand{\bfy}{\ensuremath{\mathbf{y}}}

\newcommand{\bfG}{{\bf G}}

\newcommand{\bfY}{\ensuremath{\mathbf{Y}}}

\newcommand{\calO}{\ensuremath{\mathcal{O}}}

\newcommand{\calR}{\ensuremath{\mathcal{R}}}
\newcommand{\calT}{\ensuremath{\mathcal{T}}}

\newcommand{\calZ}{\ensuremath{\mathcal{Z}}}

\newcommand{\Gri}{\ensuremath{\mathcal{O}}}

\newcommand{\dotcup}{\ensuremath{\mathbin{\mathaccent\cdot\cup}}}
\newcommand{\bigdotcup}{\ensuremath{\mathop{\dot{\bigcup}}}}

\newcommand{\lri}{\mathfrak{o}}

\newcommand{\transpose}[1]{{#1}^\mathrm{t}}

\renewcommand{\epsilon}{\varepsilon}
\renewcommand{\phi}{\varphi}

\DeclareMathOperator{\Sp}{Sp}
\DeclareMathOperator{\GL}{GL}

\DeclareMathOperator{\Mat}{Mat}

\DeclareMathOperator{\real}{Re}

\DeclareMathOperator{\Pfaff}{Pf}

\DeclareMathOperator{\Spec}{Spec}

\newcommand{\group}{Q}
\newcommand{\aalpha}{z}

\makeatother

\usepackage{babel}
  \providecommand{\lemmaname}{Lemma}
  \providecommand{\remarkname}{Remark}
\providecommand{\theoremname}{Theorem}

\begin{document}
\title[Representation zeta functions of some nilpotent
groups]{Representation zeta functions of some nilpotent groups
  associated to prehomogeneous vector spaces}

\keywords{finitely generated nilpotent groups, representation zeta
  functions, Igusa local zeta functions, $\mfp$-adic integration,
  prehomogeneous vector spaces} \subjclass[2000]{22E55, 20F69, 05A15, 11S90, 11M41}
%22E55   	Representations of Lie and linear algebraic groups over global fields and adï¿œle rings
%20F69   	Asymptotic properties of groups
%05A15   	Exact enumeration problems, generating functions
%11S90          Prehomogeneous vector spaces
%11M41   	Other Dirichlet series and zeta functions

\begin{abstract}
  We compute the representation zeta functions of some finitely
  generated nil\-po\-tent groups associated to unipotent group schemes
  over rings of integers in number fields. These group schemes are
  defined by Lie lattices whose presentations are modelled on certain
  prehomogeneous vector spaces. Our method is based on evaluating
  $\mfp$-adic integrals associated to certain rank varieties of
  matrices of linear forms.
\end{abstract}
%\footnote{CV: switch off showkeys and thanks before submission.}
\author{Alexander Stasinski and Christopher Voll}

\address{Alexander Stasinski, Department of Mathematical Sciences,
  Dur\-ham University, Dur\-ham, DH1 3LE, UK}

\email{alexander.stasinski@durham.ac.uk}

\address{Christopher Voll, Fakult\"{a}t f\"{u}r Mathematik,
  Universit\"{a}t Bielefeld, Postfach 100131, D-33501 Bielefeld,
  Germany} \email{voll@math.uni-bielefeld.de}

%\thanks{This file is called \boxed{\text{\jobname}} \hfill{}\textbf{Date
%of draft version: \today}}

\maketitle
\setcounter{tocdepth}{4}
\thispagestyle{empty}
%\tableofcontents

\section{\label{sec:Introduction}Introduction}
Let $G$ be a finitely generated torsion-free nilpotent group (or
$\calT$-group, for short). The representation zeta function of $G$ is
the Dirichlet series
\[
\zeta_{G}(s):=\sum_{n=1}^{\infty}\widetilde{r}_{n}(G)n^{-s},
\]
where $\widetilde{r}_{n}(G)$ denotes the number of twist-isoclasses of
complex $n$-dimensional irreducible representations of $G$ and $s$ is
a complex variable. If $K$ is a number field with ring of integers
$\calO$ and $\bfG$ is a unipotent group scheme over $\calO$, then
$\bfG(\calO)$ is a $\calT$-group and $\zeta_{\bfG(\calO)}(s)$ has an
Euler product indexed by the non-zero prime ideals of $\calO$:
\begin{equation}\label{equ:fine.euler}
\zeta_{\bfG(\calO)}(s)=\prod_{\mfp}\zeta_{\bfG(\calO_{\mfp})}(s),
\end{equation}
where $\zeta_{\bfG(\calO_{\mfp})}(s)$ is the zeta function enumerating
twist-isoclasses of continuous irreducible representations of the
pro-$p$ group $\bfG(\calO_{\mfp})$; see
\cite[Proposition~2.2]{StasinskiVoll-RepsTgps}.  General properties of
representation zeta functions of the form $\zeta_{\bfG(\calO)}(s)$
were studied in \cite{StasinskiVoll-RepsTgps}. In particular, it was
shown there that for all but finitely many $\mfp$ the factors of the
Euler product~\eqref{equ:fine.euler} are rational functions in
$q^{-s}$, where~$q= |\calO/\mfp|$.

Almost all of the Euler factors may be expressed in terms of
$\mfp$-adic integrals associated to rank varieties of matrices of
linear forms; cf.\ \cite[Corollary~2.11]{StasinskiVoll-RepsTgps}. In
general, computing these integrals is a hard problem. In
\cite[Theorem~B]{StasinskiVoll-RepsTgps} we computed the
representation zeta functions for three infinite, explicitly described
families. One motivation for these computations was the idea to
construct and study $\calT$-groups through presentations modelled on
prehomogeneous vector spaces (PVSs) and their relative
invariants. Indeed, the study of $\mfp$-adic integrals associated to
the relative invariants of PVSs has a comparatively long history; see
\cite{Denef/91} and \cite{Kimura/03} for details. These integrals also
served as test cases for a number of far-reaching conjectures and
explicit formulae for many such integrals are known.

In the current paper as well as in
\cite[Theorem~B]{StasinskiVoll-RepsTgps}, we consider $\calT$-groups
which are groups of rational points of unipotent group schemes defined
in terms of certain $\Z$-Lie lattices. The latter are defined by
(antisymmetric) matrices of linear forms designed such that their
Pfaffians equal, or are closely related to, the relative invariants of
PVSs.

The unipotent group schemes $F_{n,\delta},G_{n},H_{n}$ defined in
\cite[Definition~1.2]{StasinskiVoll-RepsTgps} are modelled on the
first three types of irreducible PVSs listed in the appendix of
Kimura's book \cite{Kimura/03}; see also \cite[Examples 2.1, 2.2,
2.3]{Kimura/03}. For these three infinite families, we established
connections between the $\mfp$-adic integrals associated to relative
invariants of the PVSs and the representation zeta functions of the
associated groups; see \cite[Section~6]{StasinskiVoll-RepsTgps}.

In the present paper we compute the representation zeta functions of
certain $\calT$-groups connected to PVSs of the
form $$(\Sp_{m}\times\GL_{2n}, \Lambda_1\otimes\Lambda_1, V(2m)\otimes
V(2n)),$$ defined in \cite[Example~2.13]{Kimura/03}, for $m\in\{1,2\}$
and $n=1$.  Note that our use of $m$ and $n$ is opposite to that of
Kimura but consistent with Igusa~\cite[p.~165]{Igusa/00}. As Kimura
explains, for $m=n$ these PVSs are special cases of the class of PVSs
described in \cite[Example~2.1]{Kimura/03} and the cases where $m > n
> m/2$ are (castling) equivalent to those where $m \geq 2n \geq 2$, so
one may restrict to the latter. Our construction only requires $m\geq
n\geq 1$.  For $m\geq 2n\geq 2$ the PVSs are irreducible, of type (13)
in Kimura's list. After the PVSs of types (1), (2) and (3), type (13)
is the first in Kimura's list to comprise an infinite family of PVSs.

We now give details of our constructions. Let $m\geq n\geq1$ be
integers.  Let $(Y_{ij})\in\Mat_{2m,2n}(\Z[\bfY_0])$ be the generic
matrix in variables $\bfY_0=(Y_{11},\dots,Y_{2m,2n})$ and
$$J_{m}= 
\begin{pmatrix}0 & 1\\
-1 & 0\\
 &  & \ddots\\
 &  &  & 0 & 1\\
 &  &  & -1 & 0
\end{pmatrix}
\in\Mat_{2m}(\Z)$$ be the structure matrix of the standard
non-degenerate symplectic form on $\Z^{2m}$, viz.\ the matrix
$(a_{ij})$ with $a_{2i-1,2i}=1$ and $a_{2i,2i-1}=-1$ for $1\leq i \leq
m$ and all other entries equal to zero.  The relative invariant of the relevant 
PVS is the
homogeneous polynomial
$$f(\bfY_0)=\Pfaff(\transpose{(Y_{ij})}J_{m}(Y_{ij}))\in\Z[\bfY_0]$$
of degree $2n$, where $\Pfaff$ denotes the Pfaffian of an
antisymmetric matrix.  Ideally, we are looking to define an
antisymmetric matrix of linear homogeneous forms whose Pfaffian is
equal to~$f$. If $m=n$, then $f(\bfY_0)=\det((Y_{ij}))$ and the
matrix
\begin{equation}\label{equ:comm.matrix.G2}
  \left( \begin{matrix} 0 & (Y_{ij}) \\ -\transpose{(Y_{ij})} & 0 \end{matrix}\right)\in\Mat_{4n}{\Z[\bfY_0]}
\end{equation}
has the desired property. In the general case $m>n$, we have not been
able to find such a matrix, and it is conceivable that none
exists. Instead we consider a matrix whose Pfaffian is very close
to~$f$. More precisely, let $Y$ be a variable, $\bfY=(Y,\bfY_0)$,
and define
\begin{equation}\label{equ:comm.matrix}
\calR(\bfY) = \begin{pmatrix}YJ_{m} & (Y_{ij})\\
  -\transpose{(Y_{ij})} & 0
\end{pmatrix} \in \Mat_{2(m+n)}(\Z[\bfY]).
\end{equation}
Then $\Pfaff(\calR(\bfY))=Y^{m-n}f(\bfY_0)$; see
Lemma~\ref{lem:Pfaffian}. The matrix $\calR(\bfY)$ defines the
class-$2$-nilpotent $\Z$-Lie lattice
\begin{multline*}
  \mathcal{Q}_{m,n} = \langle x_{r},y_{ij},y \mid
  1 \leq r\leq 2(m+n), 1 \leq i \leq 2m, 1 \leq j \leq 2n, \\
  \forall\, 1 \leq a,b \leq 2(m+n): \;[x_a,x_b] = \calR(\bfy)_{a,b},
  \; y,y_{ij} \textrm{ central}\rangle.
\end{multline*}
This Lie lattice in turn defines a unipotent group scheme
$\group_{m,n}$ over $\Z$ via the Hausdorff series;
cf.~\cite[Section~2.12]{StasinskiVoll-RepsTgps}. Note that, in this
vein, the matrix~\eqref{equ:comm.matrix.G2} defines the $\Z$-Lie
lattice $\mathcal{G}_2$ in
\cite[Definition~1.2]{StasinskiVoll-RepsTgps}.

Let $K$ be a number field with ring of integers $\calO$. Then
$\group_{m,n}(\calO)$ is a $\calT$-group of nilpotency class $2$.  For
a non-zero prime ideal $\mfp$ of $\calO$, we write $\lri=\calO_{\mfp}$
for the completion of $\calO$ at $\mfp$ and $q=|\calO_{\mfp}/\mfp|$
for the residue cardinality. Let $\zeta_K(s)$ denote the Dedekind zeta
function of $K$.  Our main results are the following.

\begin{thm}\label{thm:m1n1}
For every non-zero prime ideal $\mfp$ of $\calO$, writing $t=q^{-s}$,
  $$
    \zeta_{\group_{1,1}(\calO_\mfp)}(s)= \frac{(1-t)(1-q^{2}t)}{
      (1-q^{3}t)(1-q^{4}t)}.
  $$
\end{thm}

\begin{corollary}\label{cor:m1n1}
  The zeta function
  $$\zeta_{\group_{1,1}(\calO)}(s)=
  \frac{\zeta_K(s-3)\zeta_K(s-4)}{\zeta_K(s)\zeta_K(s-2)}$$ has
  abscissa of convergence $5$ and meromorphic continuation to the
  whole complex plane.
\end{corollary}

\begin{thm}\label{thm:m2n1}
For every non-zero prime ideal $\mfp$ of $\calO$, writing $t=q^{-s}$,
$$
\zeta_{\group_{2,1}(\calO_\mfp)}(s) =
\frac{(1-t)(1+q^{3}t)(1-q^{2}t)(q^{8}t^{3}-q^{7}t^{2}+q^{6}t^{2}-q^{5}t^{2}-q^{3}t+q^{2}t-qt+1)}{(1-q^{5}t)(1-q^{5}t^{2})(1-q^{4}t)(1-q^8t^2)}.
$$
\end{thm}

\begin{corollary}\label{cor:m2n1}
  The zeta function $\zeta_{\group_{2,1}(\calO)}(s)$ has abscissa of
  convergence $6$ and meromorphic continuation to $\{ s \in \mathbb{C}
  \mid \mathrm{Re}(s) > 7/2\}$.
\end{corollary}
Our earlier result \cite[Theorem~B]{StasinskiVoll-RepsTgps} comprises
an explicit computation of the representation zeta functions of
infinite families of groups. Its proof is based on a recursive
procedure, exploiting the genericity of various matrices of linear
forms encoding presentations of the relevant Lie lattices. In the
present case, the matrices $\calR(\bfY)$ do not seem to lend
themselves to a similar recursive analysis. This is why we resort to
an explicit analysis of $\mfp$-adic integrals associated to the
relevant rank varieties.

We remark that the uniformity of the analytic invariants determined in
Corollaries~\ref{cor:m1n1} and \ref{cor:m2n1}, viz.\ their
independence of $\calO$, is a general feature of representation zeta
functions of $\T$-groups obtained from unipotent group schemes; cf.\
\cite{DuongVoll/15} for details.

\subsection{Topological representation zeta functions}
In \cite{Rossmann/15}, Rossmann introduced the topological
representation zeta function $\zeta_{\bfG, \textup{top}}(s)$
associated to a unipotent group~$\bfG$ defined over a number
field. This is a rational function in a parameter $s$ which, in a
certain precise sense, captures the behaviour of the local
representation zeta functions $\zeta_{\bfG(\Gri_{\mfp})}(s)$
associated to $\bfG$ in the limit as `$q\rightarrow 1$';
cf.\ \cite[Definition~3.5]{Rossmann/15}. Informally, the topological
representation zeta function of $\bfG$ is the leading term of the
expansion in $q-1$ of the local representation zeta
functions~$\zeta_{\bfG(\Gri_{\mfp})}(s)$. Given the explicit formulae
in Theorems~\ref{thm:m1n1} and~\ref{thm:m2n1} it follows easily that
\begin{align*}
  \zeta_{Q_{1,1,\textrm{top}}}(s) = \frac{s(s-2)}{(s-3)(s-4)}\quad
  \textrm{ and } \quad \zeta_{Q_{2,1,\textrm{top}}}(s) = \frac{2 s
    (s-2) (s^2 - 5s + 5)}{(s-5)(2s-5)(s-4)^2};
\end{align*}
cf.\ \cite[Section~4]{Rossmann_stab/15}. Simple computations show that
Questions 7.1, 7.2, 7.4, and 7.5 raised in
\cite[Section~7]{Rossmann/15} have positive answers in the cases under
consideration.

\subsection{A potential connection with Coxeter group
  statistics}\label{subsec:coxeter}

A general theme in
\cite{StasinskiVoll-RepsTgps} is the description of local factors of
representation zeta functions of certain $\T$-groups in terms of
statistics on Coxeter groups of type $B$ generalizing the classical
Coxeter length function. Whenever available, such descriptions afford
much more concise and explicit formulae than those given in general by
$\mfp$-adic integrals. They also allow for direct proofs of certain
local functional equations;
cf.\ \cite[Theorem~A]{StasinskiVoll-RepsTgps}. Indeed, whilst in
general these symmetries are consequences of the Weil conjectures for
smooth projective algebraic varieties over finite fields, in the
presence of Coxeter group theoretic interpretations they are often
easily deduced from comparatively elementary symmetry features of
Coxeter groups.

We record here an observation that might help explain the
``exceptional'' factor in the numerator in Theorem~\ref{thm:m2n1} in
such Coxeter group theoretic terms.  We write $B_m$ for the Coxeter
group of signed $m \times m$-permutation matrices, $\ell(w)$ for the
Coxeter length of an element $w\in B_m$, and $D(w)$ for its (right)
descent set; cf.\ \cite[Section~8.1]{BjoernerBrenti/05} for
details. Set $Z_0 = q^4t$, $Z_1 = q^8t^2$. Then, in $\Z[q,t]$,
\begin{equation*}
  \sum_{w\in B_2} (-q^{-1})^{\ell(w)}\prod_{i\in D(w)}Z_i = 
  q^{8}t^{3}-q^{7}t^{2}+q^{6}t^{2}-q^{5}t^{2}-q^{3}t+q^{2}t-qt+1.
\end{equation*}
By \cite[Lemma~4.4]{StasinskiVoll-RepsTgps}, this means that, in
$\Q(q,t)$,
\begin{equation}\label{equ:cox1}
 \frac{q^{8}t^{3}-q^{7}t^{2}+q^{6}t^{2}-q^{5}t^{2}-q^{3}t+q^{2}t-qt+1}{(1-q^4t)(1-q^8t^2)}
 = \sum_{I \subseteq \{0,1\}} f_{2,I}(-q^{-1}) \prod_{i\in D(w)}
 \frac{Z_i}{1-Z_i},
\end{equation}
where $f_{n,I}(X) = \sum_{w\in B_n,\, D(w) \subseteq I}
X^{\ell(w)}\in\Z[X]$ for $n\in\N$ and
$I\subseteq\{0,1,\dots,n-1\}$. Specifically,
\begin{align*}
 f_{2,\varnothing} (X) & = 1,\\ f_{2,{\{0\}}}(X) = f_{2,{\{1\}}}(X) &= 1 + X + X^2
 + X^3,\\ f_{2,{\{0,1\}}}(X) &= 1 + 2X + 2X^2 + 2X^3 + X^4 \quad
 (\textrm{the Poincar\'e polynomial of $B_2$}).
\end{align*}
The contribution to the topological representation zeta function
$\zeta_{Q_{2,1,\textup{top}}}(s)$ from the factor \eqref{equ:cox1} is
$\frac{2 s (s-2) (s^2 - 5s + 5)}{(s-4)^2}$.

For $(m,n)=(1,1)$, there is a similar interpretation of a factor in
Theorem~\ref{thm:m1n1} in terms of $B_1(\cong C_2)$:
\begin{equation}\label{equ:cox2}
 \frac{1-q^2t}{1-q^3t} = 1 + (-q^{-1}+1)\frac{q^3t}{1-q^3t} =
 \sum_{I\subseteq\{0\}} f_{1,I}(-q^{-1})\frac{q^3t}{1-q^3t}.
\end{equation}
The contribution to the topological representation zeta function
$\zeta_{Q_{2,1,\textup{top}}}(s)$ from the factor \eqref{equ:cox2} is
$\frac{s-2}{s-3}$.

 Note that the
identities \eqref{equ:cox1} and \eqref{equ:cox2} are interpretations
of our explicit formulae for the respective local zeta functions. Any
\emph{a priori} proofs for these formulae seem likely to yield proofs
of Theorems~\ref{thm:m1n1} and \ref{thm:m2n1} which are much less
arduous than the $\mfp$-adic computations that make up the bulk of the
current paper. Whether analogous identities help to explain the
remaining factors and the (presumably involved) formulae for the zeta
functions $\zeta_{\group_{m,n}(\lri)}(s)$ for further values of $m$
and $n$ is therefore an interesting open question.

\subsection{Organisation and notation}\label{subsec:Notation}

Section~\ref{sec:lemmas} comprises a number of elementary preliminary
lemmas. In Section~\ref{sec:m1-n1} we compute
$\zeta_{\group_{m,n}(\lri)}(s)$ in the case $(m,n)=(1,1)$. In
Section~\ref{sec:m2-n1} we do the same for $(m,n)=(2,1)$, this case
being substantially more involved.

We fix some additional notation. Let $\N$ denote the set of positive
integers and $\N_{0}=\N\cup\{0\}$.  Let $v$ be the valuation on $\lri$
and for $x\in\lri$, let $\vert x\vert=q^{-v(x)}$ denote the
$\mfp$-adic absolute value.  For $d\in\N$ we write $\mfp^{d}$ for the
$d$th power of the ideal $\mfp$ and $\mfp^{(d)}$ for the $d$-fold
Cartesian product of $\mfp$, considered as a subset of the $d$-fold
product $\lri^{(d)}=\lri^{d}$.  We write $\lri^\times$ for the group
of units of $\lri$ and $\varpi$ for a fixed uniformiser of~$\lri$. We
set $W_{d}(\lri)=\lri^{d}\setminus\mfp^{(d)}$ and $A_d$ for a fixed
set of representatives of the residue classes under the reduction mod
$\mfp$ map $W_d(\lri)\rightarrow \F_q^d\setminus\{0\}$. Moreover, let
$A_d^0=A_d\cup\{0\}$.

%Variables of polynomials are denoted by capital letters, whereas
%(integration) variables of polynomial functions are denoted by small
%letters. 
We use boldface letters, such as $\bfY$ and $\bfy$, for vectors
of variables.  For a finite set $H$ of polynomial functions
corresponding to polynomials in $\lri[\bfY]$, let
$\|H(\bfy)\|:\lri^d\rightarrow\R$ denote the function defined by $
\bfa\mapsto\max\{\lvert f(\bfa)\rvert\mid f\in H\}.  $ If
$H=\{f_1(\bfY),\dots,f_n(\bfY)\}$ we will often write
$\|f_1(\bfy),\dots,f_n(\bfy)\|$ for $\|H(\bfy)\|$.

\subsection{Acknowledgements}
Voll acknowledges support by the DFG through Sonderforschungsbereich
701 at Bielefeld University and helpful conversations with Jan
Schepers. This research was supported by EPSRC grant EP/F044194/1. The
remarks of two anonymous referees helped to improve this paper's
exposition.

\section{Auxiliary lemmas}\label{sec:lemmas}
\begin{lem}\label{lem:Pfaffian}
  Let $m\geq n\geq 1$ and $\calR(\bfY)\in \Mat_{2(m+n)}(\Z[\bfY])$ as
  in \eqref{equ:comm.matrix}.
  Then $$\Pfaff(\calR(\bfY))=Y^{m-n}\Pfaff(\transpose{(Y_{ij})}J_{m}(Y_{ij})).$$
\end{lem}
\begin{proof}
  Consider $\calR(\bfY)$ as a matrix over the bigger ring
  $\Z[\bfY,Y^{-1}]$, so that the block $YJ_m$ is invertible. It is
  well known that if $\begin{pmatrix}A & B\\ C & D\end{pmatrix}$ is a
  block matrix over a commutative ring such that $A$ is invertible,
  then
$$
\det\begin{pmatrix}A & B\\ C & D\end{pmatrix}=\det(A)\det(D-CA^{-1}B).
$$
This implies that 
\begin{multline*}
  \det(\calR(\bfY))=\det(YJ_m)\det(-\transpose{(Y_{ij})}Y^{-1}(-J_m)(Y_{ij}))=\\Y^{2m}\det(Y^{-1}\transpose{(Y_{ij})}J_m(Y_{ij}))=Y^{2m-2n}\det(\transpose{(Y_{ij})}J_m(Y_{ij})),
\end{multline*}
whence the expression for the Pfaffian.
\end{proof}
\begin{lem}
  \textup{\label{lem:Identities} The following identities hold in the
    field $\Q((a,b,c))$:}\end{lem}
\begin{enumerate}
\item \label{enu:Identity-1}$\sum_{\substack{X\in\N}
}Xa^{X}=\frac{a}{(1-a)^{2}},$
\item \label{enu:AKOV-Identity}$\sum_{(X,Y)\in\N^{2}}a^{X}b^{Y}c^{\min\{X,Y\}}=\frac{abc(1-ab)}{(1-abc)(1-a)(1-b)}$,
\item \label{enu:Identity-2}$\sum_{(X,Y,Z)\in\N^{3}}a^{X}b^{Y+Z}c^{\min\{X,Y+Z\}}=\frac{ab^{2}c(1-a+ac-2abc+a^{2}b^{2}c)}{(1-abc)^{2}(1-a)(1-b)^{2}}.$\end{enumerate}
\begin{proof}
For \eqref{enu:Identity-1}, just observe that
\[
\sum_{\substack{X\in\N} }Xa^{X}=a\Big(\sum_{\substack{X\in\N}
}Xa^{X-1}\Big)=a\frac{d}{da}\Big(\sum_{X\in\N}a^{X}\Big)=a\frac{d}{da}\Big(\frac{a}{1-a}\Big)=\frac{a}{(1-a)^{2}}.
\]

For \eqref{enu:AKOV-Identity}, first consider the case $X\leq Y$
and let $Y=X+\tilde{Y}$ with $\tilde{Y}\in\N_{0}$. Then
\[
\sum_{\substack{(X,Y)\in\N^{2}\\
X\leq Y
}
}a^{X}b^{Y}c^{\min\{X,Y\}}=\sum_{(X,\tilde{Y})\in\N\times\N_{0}}a^{X}b^{X+\tilde{Y}}c^{X}=\frac{abc}{1-abc}\cdot\frac{1}{1-b}.
\]
Next, consider the case $X>Y$ and let $X=Y+\tilde{X}$ with
$\tilde{X}\in\N$.  Then
\[
\sum_{\substack{(X,Y)\in\N^{2}\\
X>Y
}
}a^{X}b^{Y}c^{\min\{X,Y\}}=\sum_{(\tilde{X},Y)\in\N^{2}}a^{Y+\tilde{X}}b^{Y}c^{Y}=\frac{abc}{1-abc}\cdot\frac{a}{1-a}.
\]
Thus 
\begin{align*}
\sum_{(X,Y)\in\N^{2}}a^{X}b^{Y}c^{\min\{X,Y\}} &
%=\sum_{\substack{(X,Y)\in\N^{2}\\ X\leq Y }
%}a^{X}b^{Y}c^{\min\{X,Y\}}+\sum_{\substack{(X,Y)\in\N^{2}\\ X>Y }
%}a^{X}b^{Y}c^{\min\{X,Y\}}\\ &
=\frac{abc}{1-abc}\left(\frac{1}{1-b}+\frac{a}{1-a}\right)=\frac{abc(1-ab)}{(1-abc)(1-a)(1-b)}.
\end{align*}

To prove \eqref{enu:Identity-2} %, consider the left-hand side
%\[
%\sum_{(X,Y,Z)\in\N^{3}}a^{X}b^{Y+Z}c^{\min\{X,Y+Z\}}
%\]
%and 
we make the change of variables $Y'=Y+Z$. Then
\begin{multline}
\sum_{(X,Y,Z)\in\N^{3}}a^{X}b^{Y+Z}c^{\min\{X,Y+Z\}}  =\sum_{(X,Y')\in\N^{2}}(Y'-1)a^{X}b^{Y'}c^{\min\{X,Y'\}}\label{eq:XYZ-sum} \\
= \sum_{(X,Y')\in\N^{2}}Y'a^{X}b^{Y'}c^{\min\{X,Y'\}} - \sum_{(X,Y')\in\N^{2}}a^{X}b^{Y'}c^{\min\{X,Y'\}}.
\end{multline}
Write $\sum_{(X,Y')\in\N^{2}}Y'
a^{X}b^{Y'}c^{\min\{X,Y'\}}=f_{1}+f_{2}$, where
\[
f_{1}:=\sum_{\substack{(X,Y')\in\N^{2}\\
X\leq Y'
}
}Y'a^{X}b^{Y'}c^{\min\{X,Y'\}},\qquad f_{2}:=\sum_{\substack{(X,Y')\in\N^{2}\\
X>Y'
}
}Y'a^{X}b^{Y'}c^{\min\{X,Y'\}}.
\]
Setting $Y'=X+Y''$, for $Y''\in\N_{0}$, we get, by \eqref{enu:Identity-1}, 
\begin{multline*}
  f_{1}  =\sum_{\substack{X\in\N\\ Y''\in\N_{0} }}(X+Y'')
    a^{X}b^{Y''+X}c^{X}=\sum_{\substack{X\in\N\\ Y''\in\N_{0} }}X(abc)^{X}b^{Y''}+\sum_{\substack{X\in\N\\ Y''\in\N_{0} }
    }Y''(abc)^{X}b^{Y''}\\
%  & =\sum_{\substack{X\in\N\\
%      Y''\in\N } }X(abc)^{X}b^{Y''}+\sum_{\substack{X\in\N}
%  }X(abc)^{X}+\sum_{\substack{X\in\N\\
%      Y''\in\N } }Y''(abc)^{X}b^{Y''},
%\end{align*}
%so, by \eqref{enu:Identity-1},
%\begin{equation*}
%f_{1} =\frac{abc}{(1-abc)^{2}}\cdot \frac{b}{1-b}+\frac{abc}{(1-abc)^{2}}+\frac{%abc}{1-abc}\cdot\frac{b}{(1-b)^{2}}
%=\frac{abc(1-ab^{2}c)}{(1-abc)^{2}(1-b)^{2}}.
%\end{equation*}
=\frac{abc}{(1-abc)^{2}}\cdot
\frac{1}{1-b}+\frac{abc}{1-abc}\cdot\frac{b}{(1-b)^{2}}
=\frac{abc(1-ab^{2}c)}{(1-abc)^{2}(1-b)^{2}}.
\end{multline*}
Moreover, setting $X=Y'+X'$, for $X'\in\N$, we obtain,
using~\eqref{enu:Identity-1}, that
\[
f_{2}=\sum_{\substack{(X',Y')\in\N^{2}}
}Y'a^{Y'+X'}b^{Y'}c^{Y'}=\sum_{\substack{(X',Y')\in\N^{2}}
}Y'a^{X'}(abc)^{Y'}=\frac{abc}{(1-abc)^{2}}\frac{a}{1-a}.
\]
We conclude that
\[
\sum_{(X,Y')\in\N^{2}}Y'
a^{X}b^{Y'}c^{\min\{X,Y'\}}=f_{1}+f_{2}=\frac{abc}{(1-abc)^{2}}\left(\frac{1-ab^{2}c}{(1-b)^{2}}+\frac{a}{1-a}\right).
\]
Hence, by \eqref{eq:XYZ-sum} and \eqref{enu:AKOV-Identity}, we obtain
\begin{align*}
  \lefteqn{\sum_{(X,Y,Z)\in\N^{3}}a^{X}b^{Y+Z}c^{\min\{X,Y+Z\}}}
  \\\quad \quad &
  =\frac{abc}{(1-abc)^{2}}\left(\frac{1-ab^{2}c}{(1-b)^{2}}+\frac{a}{1-a}\right)-\frac{abc(1-ab)}{(1-abc)(1-a)(1-b)}\\ &
  =\frac{ab^{2}c(1-a+ac-2abc+a^{2}b^{2}c)}{(1-abc)^{2}(1-a)(1-b)^{2}}.\qedhere
\end{align*}
\end{proof}

Key tools in our computations are $\mfp$-adic integrals associated to
polynomials or polynomial mappings, known as Igusa's local zeta
function; see \cite{Denef/91} for a basic introduction. We will use
the following simple fact throughout: for $s\in\C$ with
$\real(s)>0$,
\begin{equation*}
\int_{\mfp}|x|^{s}d\mu(x)=\frac{(1-q^{-1})q^{-1-s}}{1-q^{-1-s}}.
\end{equation*}
Here -- and, \emph{mutatis mutandis}, in the sequel -- we write $\mu$
for the additive Haar measure on $\lri$ normalised such that
$\mu(\lri)=1$. This implies that $\mu(\mfp)=q^{-1}$. 

The following identities involving $\mfp$-adic integrals and rational
functions are understood as identities of meromorphic functions. The
relevant integrals all converge provided the real parts of their
arguments are sufficiently large.
\begin{lem}
\label{lem:Integral-xy}
%For $s,t\in\C$ with both $\real(s)>0$ and $\real(t)>0$,
\[
\int_{\mfp^{(2)}}|x|^{s}\Vert x,y\Vert^{t}d\mu(x,y) =
\frac{q^{-2-s-t}(1-q^{-2-s})(1-q^{-1})}{(1-q^{-2-s-t})(1-q^{-1-s})}.
\]
\end{lem}
\begin{proof}
  As, for $X,Y\in\N$,
$$\mu(\{(x,y)\in\mfp^{(2)}\mid v(x)=X,v(y)=Y\})=(1-q^{-1})^{2}q^{-X-Y},$$ we obtain, using Lemma~\ref{lem:Identities}\,\eqref{enu:AKOV-Identity},
\begin{align*}
  \int_{\mfp^{(2)}}|x|^{s}\Vert x,y\Vert^{t}d\mu(x,y) & =\sum_{(X,Y)\in\N^{2}}(1-q^{-1})^{2}q^{-X-Y} q^{-sX-t\min\{X,Y\}}\\
  &
  =(1-q^{-1})^{2}\sum_{(X,Y)\in\N^{2}}q^{-(1+s)X}q^{-Y}q^{-t\min\{X,Y\}}\\
  & =
  (1-q^{-1})^{2}\frac{q^{-1-s}q^{-1}q^{-t}(1-q^{-1-s}q^{-1})}{(1-q^{-1-s}q^{-1}q^{-t})(1-q^{-1-s})(1-q^{-1})},
\end{align*}
% we get
%\begin{equation*}
%\sum_{(X,Y)\in\N^{2}}q^{-(1+s)X}q^{-Y}q^{-t\min\{X,Y\}} = \frac{q^{-1-s}q^{-1}q^{-t}(1-q^{-1-s}q^{-1})}{(1-q^{-1-s}q^{-1}q^{-t})(1-q^{-1-s})(1-q^{-1})}
%\end{equation*}
and the lemma follows.\end{proof}
\begin{lem}\label{lem:Integral-x-yz}
%  For $s,t\in\C$ with both $\real(s)>0$ and $\real(t)>0$,
\begin{multline*}
  \int_{\mfp^{(3)}}|x|^{s}\|x,yz\|^{t}d\mu(x,y,z) =\\
\frac{q^{-3-s-t}(1-q^{-1})(1-q^{-1-s}+q^{-1-s-t}-2q^{-2-s-t}+q^{-4-2s-t})}{(1-q^{-2-s-t})^{2}(1-q^{-1-s})}.
\end{multline*}
\end{lem}
\begin{proof}
This is very similar to the proof of Lemma~\ref{lem:Integral-xy}, but
uses Lemma~\ref{lem:Identities}\,\eqref{enu:Identity-2}.
\end{proof}

\section{The zeta function of $Q_{1,1}(\lri)$}\label{sec:m1-n1}

Let $m=n=1$ and recall the notation from
Sections~\ref{sec:Introduction} and \ref{sec:lemmas}. In this case
\calR(\bfY) has Pfaffian 
$$f(\bfY_0):=\Pfaff(\calR(\bfY))=Y_{11}Y_{22}-Y_{12}Y_{21}.$$ 
For $0\leq j\leq2$ let
\[
F_{j}(\bfY)=\{g\mid g=g(\bfY)\text{ a principal
  \ensuremath{2j\times2j}\text{ minor of }\ensuremath{\calR(\bfY)}}\}.
\]
One readily computes that
\begin{align*}
F_{0}(\bfY) &= \{1\},\\
F_{1}(\bfY) & =\{0,Y^{2},Y_{11}^{2},Y_{12}^{2},Y_{21}^{2},Y_{22}^{2}\},\\
F_{2}(\bfY) & =\{\Pfaff(\calR(\bfY))^2\} = \{f(\bfY_0)^{2}\}.
\end{align*}
By \cite[Corollary~2.11]{StasinskiVoll-RepsTgps} we can compute
$\zeta_{\group_{1,1}(\lri)}(s)$ in terms of the $\mfp$-adic integral
\[
\calZ_{\lri}(\rho,\tau):=\int_{\mfp\times
  W_{5}(\lri)}|x|^{\tau}\prod_{j=1}^{2}\frac{\|F_{j}(\bfy)\cup
  F_{j-1}(\bfy)x^{2}\|^{\rho}}{\|F_{j-1}(\bfy)\|^{\rho}}d\mu(x,\mathbf{y}),
\]
where $\rho,\tau\in\C$. More precisely, 
\begin{equation}\label{equ:zeta=integral}
 \zeta_{\group_{1,1}(\lri)}(s)=1+(1-q^{-1})^{-1}\calZ_{\lri}(-s/2,2s-6).
\end{equation}
The integral $\calZ_{\lri}(\rho,\tau)$ is a special case of (2.8) in
\cite[Section~2.2.3]{StasinskiVoll-RepsTgps} for the case $v=0$ and
$u=2$.  Since $\|F_{0}(\bfy)\|=\|F_{1}(\bfy)\|=1$, we obtain
\begin{align}
  \calZ_{\lri}(\rho,\tau) 
=\int_{\mfp\times W_{5}(\lri)}|x|^{\tau}\|f(\bfy_0),x\|^{2\rho}d\mu(x,
\mathbf{y}).\nonumber%\label{ eq:integral-1}
\end{align}
Write $W_{5}(\lri)=D_{1}\dotcup D_{2}$, where
$D_{1}=\lri^{\times}\times\lri^{4}$, $D_{2}=\mfp\times W_{4}(\lri)$.
Thus
\begin{equation}\label{equ:partition.Z.m=1}
\calZ_{\lri}(\rho,\tau)=I_{D_{1}}+I_{D_{2}},
\end{equation}
where
\begin{align*}
  I_{D_{1}} &:=\int_{\mfp\times\lri^{\times}\times\lri^{4}}|x|^{\tau}\|f(\bfy_0),x\|^{2\rho}d\mu(x,\mathbf{y}),\\
  I_{D_{2}} &:=\int_{\mfp^{(2)}\times
    W_4(\lri)}|x|^{\tau}\|f(\bfy_0),x\|^{2\rho}d\mu(x,\mathbf{y}).
\end{align*}
In the following we compute each of the integrals $I_{D_{1}}$ and
$I_{D_{2}}$ in turn.

\subsection{Computation of $I_{D_{1}}$}
Clearly
\begin{equation*}
  I_{D_{1}} 
%  =\int_{\mfp\times\lri^{\times}\times\lri^{4}}|x|^{\tau}\|f(\bfy_0),x\|^{2\rho}d\mu(x,\mathbf{y})\\ &
  =
  (1-q^{-1})\int_{\mfp\times\lri^{4}}|x|^{\tau}\|f(\bfy_0),x\|^{2\rho}d\mu(x,\bfy_0).
\end{equation*}
Recall the notation $A_d^0$ from Section~\ref{subsec:Notation} and
set, for $\aalpha\in A_4^0$,
\begin{equation}
  I_{D_{1},\aalpha}:=(1-q^{-1})\int_{\mfp\times(\aalpha+\mfp^{(4)})}|x|^{\tau}\|f(\bfy_0),x\|^{2\rho}d\mu(x,\bfy_0).\label{eq:I_D_1,a-1}
\end{equation}
Hence
\begin{equation}\label{eq:ID1.sum.m=1}
 I_{D_1} = \sum_{\aalpha\in A_4^0} I_{D_{1},\aalpha}.
\end{equation}
In the sequel we compute each of the integrals $I_{D_{1},\aalpha}$. It
will turn out that it suffices to distinguish three cases.  Let
$\bar{\aalpha}$ denote the image of $\aalpha$ mod $\mfp$.

\subsubsection{} Suppose that %$\aalpha\in A_4^0$ is such that
$f(\aalpha)\not\equiv0\bmod\mfp$. Then $\|f(\bfy_0),x\|^{2\rho}=1$
for $\bfy_0\in\aalpha+\mfp^{(4)}$, so
\begin{align*}
I_{D_{1},\aalpha} &
=(1-q^{-1})\int_{\mfp\times(\aalpha+\mfp^{(4)})}|x|^{\tau}d\mu(x,\bfy_0)\\ &
%=(1-q^{-1})\int_{\mfp\times(\aalpha+\mfp^{(4)})}|x|^{\tau}d\mu(x,\bfy_0)\\ &
=(1-q^{-1})q^{-4}\int_{\mfp}|x|^{\tau}d\mu(x)\\ &
=\frac{(1-q^{-1})^{2}q^{-5-\tau}}{1-q^{-1-\tau}} =: I_{D_1}^{1}.
\end{align*}

\subsubsection{} 
Suppose that $f(\aalpha)\equiv0\bmod\mfp$ but~$\bar{\aalpha}\neq
0$. Near any $\lri$-point of the hypersurface defined by $f$ which
reduces to a smooth point mod $\mfp$ (i.e., any point away from the
origin) we may replace $f$ by the first, say, of four coordinate
functions $\tilde{y}_1,\tilde{y}_2,\tilde{y}_3,\tilde{y}_4$ in the
relevant integral; cf.~\cite[Section~6.1]{AKOVI/13}. Using
Lemma~\ref{lem:Integral-xy} we may thus rewrite \eqref{eq:I_D_1,a-1}
as
\begin{align*}
  I_{D_{1},\aalpha} & =(1-q^{-1})\int_{\mfp^{(5)}}|x|^{\tau}\|\tilde{y}_{1},x\|^{2\rho}d\mu(x,\tilde{\bfy})\\
  & 
=(1-q^{-1})q^{-3}\int_{\mfp^{(2)}}|x|^{\tau}\|\tilde{y}_{1},x\|^{2\rho}d\mu(x,
\tilde{y}_{1})\\
  & =\frac{q^{-5-\tau-2\rho}(1-q^{-2-\tau})(1-q^{-1})^{2}}{(1-q^{-2-\tau-2\rho}
)(1-q^{-1-\tau})} =:I_{D_1}^{2}.
\end{align*}

\subsubsection{} 
Suppose finally that $\bar{\aalpha}=0$. In this case,
\[
I_{D_{1},\aalpha}=(1-q^{-1})\int_{\mfp\times\mfp^{(4)}}|x|^{\tau}\|f(\bfy_0),x\|^{2\rho}d\mu(x,\bfy_0).
\]
We make the change of variables $\bfy_0=\varpi\bfy'_0$, that is,
$
y_{ij}=\varpi y_{ij}'
$
for $y'_{ij}\in\lri$.  The norm of the Jacobian of this change of
variables is $q^{-4}$.  As $f$ is homogeneous of degree $2$, %we get
\[
I_{D_{1},\aalpha}=(1-q^{-1})\int_{\mfp\times\lri^{4}}|x|^{\tau}\|\varpi^{2}f(\bfy'_0),x\|^{2\rho}q^{-4}d\mu(x,\bfy_0').
\]
Write $I_{D_{1},\aalpha}=I_{1}+I_{2}$, where
\begin{align}
I_{1} & :=(1-q^{-1})q^{-4}\int_{\mfp^{2}\times\lri^{4}}|x|^{\tau}\|\varpi^{2}f(\bfy'_0),x\|^{2\rho}d\mu(x,\bfy_0'),\nonumber\\
I_{2} & :=(1-q^{-1})q^{-4}\int_{\mfp\setminus\mfp^{2}\times\lri^{4}}|x|^{\tau}\|\varpi^{2}f(\bfy'_0),x\|^{2\rho}d\mu(x,\bfy_0').\nonumber
\end{align}
Consider $I_{1}$ and make the change of variables $x=\varpi^{2}x'$ for
$x'\in \lri$.  The norm of the Jacobian of this change of variables is
$q^{-2}$.  This yields
\begin{align}
I_{1}/(1-q^{-1}) &= q^{-4}\int_{\lri\times\lri^{4}}|\varpi^{2}x'|^{\tau}\|\varpi^{2}f(\bfy'_0),\varpi^{2}x'\|^{2\rho}q^{-2}d\mu(x',\bfy_0')\label{I1}\\ &
=q^{-6-2\tau-4\rho}\int_{\lri\times\lri^{4}}|x'|^{\tau}\|f(\bfy'_0),x'\|^{2\rho}d\mu(x',\bfy_0')\nonumber\\ &
%=q^{-6-2\tau-4\rho}\left(\int_{\lri^{\times}\times\lri^{4}}|x'|^{\tau}\|f(\bfy'_0),x'\|^{2\rho}d\mu(x',\bfy_0')+\frac{1}{1-q^{-1}}I_{D_{1}}\right)\nonumber\\ &
=q^{-6-2\tau-4\rho}\left(\int_{\lri^{\times}\times\lri^{4}}d\mu(x',\bfy_0')+\frac{1}{1-q^{-1}}I_{D_{1}}\right)\nonumber\\ &
=q^{-6-2\tau-4\rho}\left(1-q^{-1}+\frac{1}{1-q^{-1}}I_{D_{1}}\right).\nonumber
\end{align}
Next, consider $I_{2}$. When $x\in\mfp\setminus\mfp^{2}$, then
$\|\varpi^{2}f(\bfy'_0),x\|^{2\rho}=|x|^{2\rho} = q^{-2\rho}$, so 
\begin{equation*}
I_{2}  =(1-q^{-1})q^{-4}\int_{\mfp\setminus\mfp^{2}\times\lri^{4}}|x|^{\tau+2\rho}d\mu(x,\bfy_0')%\\
% & =(1-q^{-1})q^{-4}q^{-1-\tau-2\rho}(1-q^{-1})\\
  =(1-q^{-1})^{2}q^{-5-\tau-2\rho}.
\end{equation*}
Thus
\begin{multline*}
I_{D_{1},\aalpha}=I_{1}+I_{2}\\=(1-q^{-1})q^{-6-2\tau-4\rho}\left(1-q^{-1}+\frac
{1}{1-q^{-1}}I_{D_{1}}\right)+(1-q^{-1})^{2}q^{-5-\tau-2\rho}=:I_{D_1}^{3}.
\end{multline*}

\subsubsection{Conclusion}
Taking the above three cases together and noting that
$$| \{ \aalpha\in A^0_4 \mid f(\aalpha)\neq 0\bmod \mfp\}| = | \GL_2(\Fq)| = q(q-1)^2(q+1),$$
\eqref{eq:ID1.sum.m=1} yields
\begin{align*}
I_{D_{1}} &=  |\GL_{2}(\F_{q})| I_{D_1}^{1} + (q^{4}-|\GL_{2}(\F_{q})|-1) I_{D_1}^{2} + I_{D_1}^{3} \\
&= |\GL_{2}(\F_{q})|\frac{(1-q^{-1})^{2}q^{-5-\tau}}{1-q^{-1-\tau}}\\ 
&\quad 
+(q^{4}-|\GL_{2}(\F_{q})|-1)\frac{q^{-5-\tau-2\rho}(1-q^{-2-\tau})(1-q^{-1 
})^{2}
}{(1-q^{-2-\tau-2\rho})(1-q^{-1-\tau})}\\
 & \quad +
(1-q^{-1})q^{-6-2\tau-4\rho}\left(1-q^{-1}+\frac{1}{1-q^{-1}}I_{D_{1}}
\right)+(1-q^{-1})^{2}q^{-5-\tau-2\rho}.
\end{align*}
Solving for $I_{D_{1}}$, we obtain
\begin{multline*}
I_{D_{1}}=  \frac{(1-q^{-1})^{2}q^{-5-\tau}}{1-q^{-6-2\tau-4\rho}}\cdot \\\Bigg(\frac{q(q-1)^{2}(q+1)}{1-q^{-1-\tau}}+\frac{q^{-2\rho}(q^{3}+q^{2}-q-1)(1-q^{-2-\tau})}{(1-q^{-2-\tau-2\rho})(1-q^{-1-\tau})}+q^{-1-\tau-4\rho}+q^{-2\rho}\Bigg).
\end{multline*}

\subsection{Computation of $I_{D_{2}}$}

Recall that $F_2(\bfY) = \{f(\bfY_0)^2\}$. Hence
\begin{align}\label{eq:ID2-small}
I_{D_{2}} & = \int_{\mfp\times\mfp\times 
W_{4}(\lri)}|x|^{\tau}\|f(\bfy_0),x\|^{2\rho}d\mu(x,y,\bfy_0)\\
& = 
q^{-1}\int_{\mfp\times\lri^{4}}|x|^{\tau}\|f(\bfy_0),x\|^{2\rho}d\mu(x, 
\bfy_0)\nonumber
\\ 
&
\quad -q^{-1}\int_{\mfp\times\mfp^{(4)}}|x|^{\tau}\|f(\bfy_0),x\|^{2\rho}d\mu(x,\bfy_0)\nonumber
\\ 
& = \frac{q^{-1}}{1-q^{-1}}I_{D_{1}}-q^{-1}\int_{\mfp\times\mfp^{(4)}}|x|^{\tau}\|f(\bfy_0),x\|^{2\rho}d\mu(x,\bfy_0
).\nonumber
\end{align}
To compute the integral
\[
I:=\int_{\mfp\times\mfp^{(4)}}|x|^{\tau}\|f(\bfy_0),x\|^{2\rho}d\mu(x,\bfy_0)
\]
we write it as $I=J_{1}+J_{2}$, where
\begin{align*}
J_{1}  &:=\int_{\mfp\setminus\mfp^{2}\times\mfp^{(4)}}|x|^{\tau}\|f(\bfy_0),x\|^{2\rho}d\mu(x,\bfy_0),\\
J_{2}  &:=\int_{\mfp^{2}\times\mfp^{(4)}}|x|^{\tau}\|f(\bfy_0),x\|^{2\rho}d\mu(x,\bfy_0).
\end{align*}
Clearly
\begin{equation*}
J_{1} =  \int_{\mfp\setminus\mfp^{2}\times\mfp^{(4)}}|x|^{\tau+2\rho}d\mu(x,\bfy_0)
  = q^{-4}\int_{\mfp\setminus\mfp^{2}}|x|^{\tau+2\rho}d\mu(x)
  =  q^{-5-\tau-2\rho}(1-q^{-1}).
\end{equation*}
Next, consider $J_{2}$ and make the change of variables $\bfy_0=
\varpi \bfy_0'$ for $\bfy_0\in \lri^4$. The norm of the Jacobian
of this change of variables is $q^{-4}$, whence, using \eqref{I1},
\begin{multline*}J_2 = \int_{\mfp^{(2)}\times\lri^{4}}|x|^{\tau}\|\varpi^{2}f(\bfy'_0), x\|^{2\rho}q^{-4}d\mu(x',\bfy_0')=\\\frac{I_1}{1-q^{-1}}=q^{-6-2\tau-4\rho}\left((1-q^{-1})+\frac{I_{D_{1}}}{1-q^{-1}}\right).
\end{multline*}

Hence, by \eqref{eq:ID2-small}, we get
\begin{align*}
I_{D_{2}} 
&=\frac{q^{-1}}{1-q^{-1}}I_{D_{1}}-q^{-1}(J_{1}+J_{2})\\ 
&=\frac{q^{-1}}{1-q^{-1}}I_{D_{1}}-q^{-1}\left(q^{-5-\tau-2\rho}(1-q^{
-1} )+q^{-6-2\tau-4\rho}\left((1-q^{-1})+\frac{I_{D_{1}}}{1-q^{-1}}\right)\right)\\
 &=\frac{q^{-1}(1-q^{-6-2\tau-4\rho})I_{D_{1}}}{1-q^{-1}}-q^{
-6-\tau-2\rho}(1-q^{-1})(1+q^{-1-\tau-2\rho}).
\end{align*}

\subsection{Conclusion}
With the computations of the two previous sections, equation
\eqref{equ:partition.Z.m=1} implies that, for $a=q^{-1}$,
$b=q^{-\tau}$, $c=q^{-\rho}$,
\begin{align*}
\calZ_{\lri}(\rho,\tau) &= I_{D_{1}}+I_{D_{2}}\\
&= 
\frac{ab(a-1)^{2}}{(1-{a}^{3}b{c}^{2})(1-{a}^{2}b{c}^{2})(1-ab)}\\
&\quad\Big({a}^{9}{b}^{2}{c}^{4}+{a}^{8}{b}^{2}{c}^{4}+{a}^{7}{b}^{2}{c}^{4}-{a}
^ { 7 } b 
{c}^{4}+{a}^{6}{b}^{2}{c}^{4}-2\,{a}^{6}b{c}^{4}+{a}^{5}{b}^{2}{c}^{4}-{a}^{5}b{ 
c}^{4}-{a}^{4}b{c}^{4}\\
&\quad-{a}^{5}b{c}^{2}-{a}^{4}b{c}^{2}+{a}^{4}{c}^{2}-2{a}^{3}b{c}^{2}+{a}^{3
}{c}^{2}-{a}^{2}b{c}^{2}+2\,{a}^{2}{c}^{2}+a{c}^{2}-{a}^{2}+1\Big)
\end{align*}
With \eqref{equ:zeta=integral} and a direct computation, this
completes the proof of Theorem~\ref{thm:m1n1}.
\begin{rem}
  As in \cite[Section~6]{StasinskiVoll-RepsTgps}, it is interesting to
  compare $\zeta_{Q_{1,1}}(s)$ to the Igusa integral of the relative
  invariant of the corresponding PVS. By \cite[p.~165]{Igusa/00}, with
  $t=q^{-s}$,
\[
Z_{1,1}(s):=\int_{\Mat_2(\mathfrak{o})}|f(\bfy_0)|^{s}\,
d\mu(\bfy_0) = \frac{(1-q^{-2})(1-q^{-1})}{(1-q^{-2}t)(1-q^{-1}t)}.
\]
We note that the real parts of the poles of $\zeta_{\group_{1,1}(\lri)}(s)$
are given by an additive translation of the real parts of the poles
of $Z_{1,1}(s)$, just as for the groups $F_{n,\delta}(\lri),G_{n}(\lri),H_{n}(\lri)$
considered in \cite{StasinskiVoll-RepsTgps}.
\end{rem}

\section{The zeta function of $\group_{2,1}(\lri)$}\label{sec:m2-n1}
Let $m=2$, $n=1$, and recall the notation from Sections~\ref{sec:Introduction}
and \ref{sec:lemmas}. In this case the Pfaffian of $\calR(\bfY)$ is
$\Pfaff(\calR(\bfY))=Y h(\bfY_0)$,
where $$h(\bfY_0)=Y_{11}Y_{22}-Y_{12}Y_{21}+Y_{31}Y_{42} -
Y_{32}Y_{41}.$$ For $0\leq j\leq3$ let
\[
F_{j}(\bfY)=\{g\mid g=g(\bfY)\text{ a principal
  \ensuremath{2j\times2j}\text{ minor of }\ensuremath{\calR({\bf Y})}}\}.
\]
 Given
integers $r,s$ such that $1\leq r<s\leq4$, we write
$m_{rs}=m_{rs}(\bfY_0)$ for the $2\times2$-minor of
$(\bfY_0)\in\Mat_{4\times 2}(\Z[\bfY_0])$ indexed by the $r$th and
$s$th rows. Note that $h = m_{12} + m_{34}$. One readily computes that
\begin{align*}
  F_{0}(\bfY) &=  \{1\},\\
  F_{1}(\bfY) &= \{0,Y^{2},Y_{ij}^{2}\mid i\in\{1,2,3,4\},j\in\{1,2\}\},\\
  F_{2}(\bfY) &= \{Y^{4},Y^{2}Y_{ij}^{2},m_{rs}(\bfY_0)^{2}\mid
  i\in\{1,2,3,4\},j\in\{1,2\},1\leq r<s\leq4\},\\
  F_{3}(\bfY) &= \{ \Pfaff(\calR(\bfY))^2\} =
  \{Y^{2}h(\bfY_0)^{2}\}.
\end{align*}
By \cite[Corollary~2.11]{StasinskiVoll-RepsTgps} we can compute
$\zeta_{\group_{2,1}(\lri)}(s)$ in terms of the $\mfp$-adic integral
\[
\calZ_{\lri}(\rho,\tau) := \int_{\mfp\times
  W_{9}(\lri)}|x|^{\tau}\prod_{j=1}^{3}\frac{\|F_{j}(\bfy)\cup
  F_{j-1}(\bfy)x^{2}\|^{\rho}}{\|F_{j-1}(\bfy)\|^{\rho}}d\mu(x,\mathbf{y}),
\]
where $\rho,\tau\in\C$. More precisely,
\begin{equation}\label{equ:zeta=integral.m2n1}
\zeta_{\group_{2,1}(\lri)}(s)=1+(1-q^{-1})^{-1}\calZ_{\lri}(-s/2,3s-10).
\end{equation}
The integral $\calZ_{\lri}(\rho,\tau)$ is a special case of (2.8) in
\cite[Section~2.2.3]{StasinskiVoll-RepsTgps} for the case $v=0$
and~$u=3$. Since $\|F_{0}(\bfy)\|=\|F_{1}(\bfy)\|=1$, we obtain
\begin{equation}
  \calZ_{\lri}(\rho,\tau)=\int_{\mfp\times
    W_{9}(\lri)}|x|^{\tau}\|F_{2}(\bfy),x^{2}\|^{\rho}\frac{\|F_{3}(\bfy)\cup
    F_{2}(\bfy)x^{2}\|^{\rho}}{\|F_{2}(\bfy)\|^{\rho}}d\mu(x,\mathbf{y}).\nonumber
\end{equation}
Write $W_{9}(\lri)=D_{1}\dotcup D_{2}$, where
$D_{1}=\lri^{\times}\times\lri^{8}$ $D_{2}=\mfp\times W_{8}(\lri)$.
Thus
\begin{equation}\label{equ:partition.Z.m=2}
 \calZ_{\lri}(\rho,\tau)=I_{D_{1}}+I_{D_2},
\end{equation}
where
\begin{align*}
 I_{D_{1}} &:=
 \int_{\mfp\times\lri^{\times}\times\lri^{8}}|x|^{\tau}\|F_{2}(\bfy),x^{2}\|^{\rho}\frac{\|F_{3}(\bfy)\cup
   F_{2}(\bfy)x^{2}\|^{\rho}}{\|F_{2}(\bfy)\|^{\rho}}d\mu(x,\mathbf{y}),\\ I_{D_2}
 & := \int_{\mfp\times\lri^{\times}\times
   W_8(\lri)}|x|^{\tau}\|F_{2}(\bfy),x^{2}\|^{\rho}\frac{\|F_{3}(\bfy)\cup
   F_{2}(\bfy)x^{2}\|^{\rho}}{\|F_{2}(\bfy)\|^{\rho}}d\mu(x,\mathbf{y}).
\end{align*}

In the following we compute each of the integrals $I_{D_{1}}$ and
$I_{D_{2}}$ in turn.  We will need to refer to the following
determinantal varieties (schemes over $\Z$):
\begin{align*}
V_{2} &= \Spec\Z[\bfY_0]/(m_{rs}(\bfY_0),\,1\leq r<s\leq4),\\
V_{3} &= \Spec\Z[\bfY_0]/(h(\bfY_0)).
\end{align*}
Clearly $V_{2}$ is a closed subscheme of $V_{3}$ and each fibre of
$V_{2}$ can be interpreted as the space of $4\times 2$-matrices of
rank at most $1$ over a field. Similarly, each fibre of $V_{3}$ can be
seen as the space of $4\times 2$-matrices or rank at most $1$ over a
field. It is well known that $V_2$ has codimension~$3$.

\subsection{Computation of $I_{D_{1}}$}

Note that $\|F_{2}(\bfy)\|=1$ and
$\|F_{3}(\bfy)\|=\|h(\bfy_0)^2\|$ as $y\in\lri^{\times}$.  Thus
\begin{align*}
  I_{D_{1}} = &\int_{\mfp\times\lri^{\times}\times\lri^{8}}|x|^{\tau}\|F_{3}(\bfy),x^{2}\|^{\rho}d\mu(x,\mathbf{y})\\
  = &\int_{\mfp\times\lri^{\times}\times\lri^{8}}|x|^{\tau}\|h(\bfy_0),x\|^{2\rho}d\mu(x,\mathbf{y})\\
  =
  &(1-q^{-1})\int_{\mfp\times\lri^{8}}|x|^{\tau}\|h(\bfy_0),x\|^{2\rho}d\mu(x,\bfy_0).
\end{align*}
We set, for $\aalpha\in A_8^0$,
\begin{equation}\label{eq:I_D_1,a}
  I_{D_{1},\aalpha}:=(1-q^{-1})\int_{\mfp\times(\aalpha+\mfp^{(8)})}|x|^{\tau}\%|h(\bfy_0),x\|^{2\rho}d\mu(x,\bfy_0).
\end{equation}
Hence
\begin{equation}\label{equ:ID1.sum.m=2}
 I_{D_1} = \sum_{\aalpha\in A_8^0} I_{D_{1},\aalpha}.
\end{equation}
In the sequel we compute each of the integrals $I_{D_{1},\aalpha}$. It
will turn out that it suffices to distinguish three cases.  

\subsubsection{} Suppose that $\bar{\aalpha}\protect\not\in
V_{3}(\F_{q})$.  In this case, $\|h(\bfy_0),x\|^{2\rho}=1$ for
$\bfy_0\in\aalpha+\mfp^{(8)}$, so
\begin{align*}
  I_{D_{1},\aalpha} &=(1-q^{-1})\int_{\mfp\times(\aalpha+\mfp^{(8)})}|x|^{\tau}d\mu(x,\bfy_0)\\
  &= (1-q^{-1})q^{-8}\int_{\mfp}|x|^{\tau}d\mu(x) =
  \frac{(1-q^{-1})^{2}q^{-9-\tau}}{1-q^{-1-\tau}} =: I_{D_1}^{1}.
\end{align*}

\subsubsection{} Suppose that $\bar{\aalpha}\in
V_{3}(\F_{q})\setminus\{0\}$.  Near any $\lri$-point of $V_3$ which
reduces to a smooth point mod $\mfp$ (i.e, a point away from the
origin), we may replace $h$ by the first, say, of eight coordinate
functions $\tilde{y}_{ij}$, where $i\in \{1,\dots,4\}$, $j\in\{1,2\}$,
in the integral~\eqref{eq:I_D_1,a}.  It may thus be rewritten, using
Lemma~\ref{lem:Integral-xy}, as
\begin{align*}
  I_{D_{1},\aalpha} 
&=(1-q^{-1})\int_{\mfp^{(9)}}|x|^{\tau}\|\tilde{y}_{11},x\|^{2\rho}d\mu(x,
\tilde{\bfy}_0)\\ 
&=(1-q^{-1})q^{-7}\int_{\mfp^{(2)}}|x|^{\tau}\|\tilde{y}_{11},x\|^{2\rho}d\mu(x, 
\tilde{y}_{1})\\
  &= 
\frac{q^{-9-\tau-2\rho}(1-q^{-2-\tau})(1-q^{-1})^{2}}{(1-q^{-2-\tau-2\rho})(1-q^
{-1-\tau})} =: I_{D_1}^{2}.
\end{align*}

\subsubsection{}
Suppose finally that $\bar{\aalpha}=0$.  In this case,
\[
I_{D_{1},\aalpha}=(1-q^{-1})\int_{\mfp^{(9)}}|x|^{\tau}\|h(\bfy_0),x\|^{2\rho}d\mu(x,\bfy_0).
\]
Make the change of variables $\bfy_0=\varpi\bfy'_0$, that is, $
y_{ij}=\varpi y_{ij}', $ for $y'_{ij}\in\lri$.  The norm of the
Jacobian of this change of variables is $q^{-8}$.  Since $h$ is
homogeneous of degree $2$,
\[
I_{D_{1},\aalpha}=(1-q^{-1})\int_{\mfp\times\lri^{8}}|x|^{\tau}\|\varpi^{2}h(\bfy'_0),x\|^{2\rho}q^{-8}d\mu(x,\bfy_0').
\]
Write $I_{D_{1},\aalpha}=I_{1}+I_{2}$, where
\begin{align*} I_{1}
  &:=(1-q^{-1})q^{-8}\int_{\mfp^{2}\times\lri^{8}}|x|^{\tau}\|\varpi^{2}h(\bfy'_0),x\|^{2\rho}d\mu(x,\bfy_0'),\\ I_{2}
  &:=(1-q^{-1})q^{-8}\int_{\mfp\setminus\mfp^{2}\times\lri^{8}}|x|^{\tau}\|\varpi^{2}h(\bfy'_0),x\|^{2\rho}d\mu(x,\bfy_0').
\end{align*}
Consider $I_{1}$ and make the change of variables $x=\varpi^{2}x'$ for
$x'\in\lri$. The norm of the Jacobian of this change of variables is
$q^{-2}$. This yields
\begin{align*}
I_{1}/(1-q^{-1}) &= q^{-8}\int_{\lri\times\lri^{8}}|\varpi^{2}x'|^{\tau}\|\varpi^{2}h(\bfy'_0),\varpi^{2}x'\|^{2\rho}q^{-2}d\mu(x',\bfy_0')\\
 &=q^{-10-2\tau-4\rho}\int_{\lri\times\lri^{8}}|x'|^{\tau}\|h(\bfy'_0),x'\|^{2\rho}d\mu(x',\bfy_0')\\
&=q^{-10-2\tau-4\rho}\left(\int_{\lri^{\times}\times\lri^{8}}d\mu(x',\mathbf{y}_
  {ij}')+\frac{1}{1-q^{-1}}I_{D_{1}}\right)\\
 &=q^{-10-2\tau-4\rho}\left(1-q^{-1}+\frac{1}{1-q^{-1}}I_{D_{1}}\right).
\end{align*}
Next, consider $I_{2}$. When $x\in\mfp\setminus\mfp^{2}$, we have 
$\|\varpi^{2}h(\bfy_0),x\|^{2\rho}=|x|^{2\rho}=q^{-2\rho}$, so 
$$
I_{2}=(1-q^{-1})q^{-8}\int_{\mfp\setminus\mfp^{2}\times\lri^{8}}|x|^{\tau+2\rho}
d\mu(x,\bfy_0')=(1-q^{-1})^{2}q^{-9-\tau-2\rho}.
$$
We thus obtain
\begin{multline*}
  I_{D_{1},\aalpha}=I_{1}+I_{2}=\\(1-q^{-1})q^{-10-2\tau-4\rho}\left(1-q^{-1}+\frac{1}{
      1-q^{-1}}I_{D_{1}}\right)+(1-q^{-1})^{2}q^{-9-\tau-2\rho} =:
  I_{D_1}^{3}.
\end{multline*}

\subsubsection{Conclusion}
Taking the above three cases together, \eqref{equ:ID1.sum.m=2}
yields
\begin{align}
I_{D_{1}} &= |\mathbb{A}^8(\Fq)\setminus V_{3}(\F_{q})| I_{D_1}^{1} + |V_{3}(\F_{q})\setminus\{0\}|I_{D_1}^{2} + I_{D_1}^3 \label{equ:ID1}\\
&= (q^8 - |V_{3}(\F_{q})|)\frac{(1-q^{-1})^{2}q^{-9-\tau}}{1-q^{-1-\tau}} +(|V_{3}(\F_{q})|-1)\frac{q^{-9-\tau-2\rho}(1-q^{-2-\tau})(1-q^{-1})^{2}}{(1-q^{-2-\tau-2\rho})(1-q^{-1-\tau})}\nonumber\\
 &\quad +(1-q^{-1})q^{-10-2\tau-4\rho}\left(1-q^{-1}+\frac{1}{1-q^{-1}}I_{D_{1}}\right)+(1-q^{-1})^{2}q^{-9-\tau-2\rho}\nonumber.
\end{align}
\begin{lem}\label{lem:V2V3-Fq}
\begin{equation*}
|V_{2}(\F_{q})|  = (q+1)(q^4-1)+1, \quad
|V_{3}(\F_{q})|  = q^{3}(q^{4}+q-1).
\end{equation*}
\end{lem}
\begin{proof}
This can be easily proved directly or read off from 
\cite[Proposition~3.1]{LaksovThorup/94}.
\end{proof}
Solving \eqref{equ:ID1} for $I_{D_{1}}$, we obtain
\begin{multline}
I_{D_{1}} =
\frac{(1-q^{-1})^{2}q^{-9-\tau}}{1-q^{-10-2\tau-4\rho}}
\\ \cdot\Bigg(\frac{(q^{8}-q^{3}(q^{4}+q-1))}{1-q^{-1-\tau}}+\frac{q^{-2\rho}(q^{3}(q^{4}+q-1)-1)(1-q^{-2-\tau})}{(1-q^{-2-\tau-2\rho})(1-q^{-1-\tau})}
+q^{-1-\tau-4\rho}+q^{-2\rho}\Bigg).\label{equ:ID1.solved}
\end{multline}

\subsection{Computation of $I_{D_{2}}$}\label{subsec:ID2.m=2}
We set, for $\aalpha\in A_8$,
\begin{equation}
\label{eq:I_D_2,a}
I_{D_{2},\aalpha}  = 
\int_{\mfp^{(2)}\times(\aalpha+\mfp^{(8)})}|x|^{\tau}\|F_{2}(\bfy),x^{2}\|^{\rho
}\frac{\|F_{3}(\bfy)\cup 
F_{2}(\bfy)x^{2}\|^{\rho}}{\|F_{2}(\bfy)\|^{\rho}}d\mu(x,\mathbf{y}).
\end{equation}
Hence
$$I_{D_2} = \sum_{\aalpha\in A_8} I_{D_2,\aalpha}.$$
In the sequel we compute each of the integrals $I_{D_2,\aalpha}$. We
will distinguish three cases, where the third one will be split
further into seven subcases.

\subsubsection{} If $\bar{\aalpha}\not\in V_{3}(\F_{q})$ and $\bfy =
(y,\bfy_0)\in\mfp\times(\aalpha+\mfp^{(8)})$, then
$|m_{rs}(\bfy_0)|=1$ for some $r<s$, and so $\|F_{2}(\bfy)\|=1$.
Moreover, $|h(\aalpha)|=1$, so $\|F_{3}(\bfy)\|=|y^{2}|$. Thus, using
Lemma~\ref{lem:Integral-xy},
\begin{align*}
I_{D_{2},\aalpha}  %&= 
%\int_{\mfp^{(2)}\times(\aalpha+\mfp^{(8)})}|x|^{\tau}\|F_{2}(\bfy),x^{2}\|^{\rho
%}\frac{\|F_{3}(\bfy)\cup 
%F_{2}(\bfy)x^{2}\|^{\rho}}{\|F_{2}(\bfy)\|^{\rho}}d\mu(x,\mathbf{y})\\
 &= q^{-8}\int_{\mfp^{(2)}}|x|^{\tau}\|x,y\|^{2\rho}d\mu(x,y)\\
 &= \boxed{\frac{q^{-10-\tau-2\rho}(1-q^{-2-\tau})(1-q^{-1})}{(1-q^{-2-\tau-2\rho}
)(1-q^{-1-\tau})}} =: I_{D_{2}}^{1}.
\end{align*}

\subsubsection{} If $\bar{\aalpha}\in V_{3}(\F_{q})\setminus
V_{2}(\F_{q})$ and $\bfy =
(y,\bfy_0)\in\mfp\times(\aalpha+\mfp^{(8)})$, then
$\|F_{2}(\bfy)\|=1$, as above. Near any $\lri$-point of $V_3$ which
reduces to a smooth point mod $\mfp$ (i.e., any point away from the
origin) we may replace $h$ by the first, say, of eight coordinate functions
$\tilde{y}_{ij}$, where $i\in \{1,\dots,4\}$, $j\in\{1,2\}$, in the
integral~\eqref{eq:I_D_2,a}. Thus, using
Lemma~\ref{lem:Integral-x-yz},
\begin{align*}
I_{D_{2},\aalpha} &= 
\int_{\mfp^{(2)}\times(\aalpha+\mfp^{(8)})}|x|^{\tau}\|F_{3}(\bfy),x^{2}\|^{\rho
}d\mu(x,\mathbf{y})=\int_{\mfp^{(10)}}|x|^{\tau}\|y\tilde{y}_{1},x\|^{2\rho}
d\mu(x,y,\tilde{\bfy}_0)\\
 &=q^{-7}\int_{\substack{\mfp^{(3)}}
}|x|^{\tau}\|y\tilde{y}_{1},x\|^{2\rho}d\mu(x,y)\\
&=\boxed{\frac{q^{-10-\tau-2\rho}(1-q^{-1})(1-q^{-1-\tau}+q^{-1-\tau-2\rho}-2q^{
-2-\tau-2\rho}+q^{-4-2\tau-2\rho})}{(1-q^{-2-\tau-2\rho})^{2}(1-q^{-1-\tau})}} 
=: I_{D_{2}}^{2}.
\end{align*}

\subsubsection{}
Finally we consider the case $\bar{\aalpha}\in V_{2}(\F_{q})$, which
will itself split into seven subcases.

If $\bar{\aalpha}\in V_{2}(\F_{q})$ and $\bfy =
(y,\bfy_0)\in\mfp\times(\aalpha+\mfp^{(8)})$, then
$|m_{rs}(\bfy_0)|<1$ for all~$r<s$. Since $\bar{\aalpha}$ is a
smooth point, we may introduce coordinates
$\tilde{\bfy}_0=(\tilde{y}_{ij})$, for $i\in\{1,\dots,4\}$,
$j\in\{1,2\}$, such that the $\lri$-points $V_{3}(\lri)$ are given by
the vanishing of the coordinate function
$\tilde{y}_{1}:=\tilde{y}_{11}$, and $V_{2}(\lri)$ is given by the
vanishing of the coordinate function $\tilde{y}_{1}$,
$\tilde{y}_{2}:=\tilde{y}_{12}$ and
$\tilde{y}_{3}:=\tilde{y}_{13}$. Note that $V_3$ has codimension one
in $\mathbb{A}^8$ and that $V_2$ has codimension two in $V_3$. Since
$\bar{\aalpha}\neq0$ as $\aalpha\in W_{8}(\lri)$, there exists at
least one $y_{ij}$ such that $|y_{ij}|=1$.  Thus
\begin{align}
  I_{D_{2},\aalpha} &= \int_{\mfp^{(2)}\times(\aalpha+\mfp^{(8)})}|x|^{\tau}\|F_{2}(\bfy),x^{2}\|^{\rho}\frac{\|F_{3}(\bfy)\cup F_{2}(\bfy)x^{2}\|^{\rho}}{\|F_{2}(\bfy)\|^{\rho}}d\mu(x,\mathbf{y})\nonumber\\%\label{eq:Integral-Case3}\\
 &= \int_{\mfp^{(10)}}|x|^{\tau}\|x,y,\tilde{y}_{1},\tilde{y}_{2},\tilde{y}_{3}\|^{
2\rho}\frac{\|xy,x\tilde{y}_{1},x\tilde{y}_{2},x\tilde{y}_{3},y\tilde{y}_{1}\|^{
2\rho}}{\|y,\tilde{y}_{1},\tilde{y}_{2},\tilde{y}_{3}\|^{2\rho}}d\mu(x,y,\tilde{
\bfy})\nonumber \\ 
&= 
q^{-5}\int_{\mfp^{(5)}}|x|^{\tau}\|x,y,\tilde{y}_{1},\tilde{y}_{2},\tilde{y}_{3 
}\|^{2\rho}\frac{\|xy,x\tilde{y}_{1},x\tilde{y}_{2},x\tilde{y}_{3},y\tilde{y}_{1 
}\|^{2\rho}}{\|y,\tilde{y}_{1},\tilde{y}_{2},\tilde{y}_{3}\|^{2\rho}}d\mu(x,y, 
\tilde{\bfy}) =: I_{D_{2}}^{3}.\nonumber
\end{align}
We partition the domain of integration $\mfp^{(5)}$ into pieces on
which the above integral simplifies. %Let 
%$H=\{x_1,x_2,x_3,x_4,x_5\}$ be the set of polynomial functions 
%corresponding to the polynomials $X_1,X_2,X_3,X_4,X_5\in 
%\lri[X_1,X_2,X_3,X_4,X_5]$, and for $\bfa\in\mfp^{(5)}$, define
%$$M(\bfa)=\{x_i\in H \mid |x_i(\bfa)|=\|H(\bfa)\|\}.$$ 
For $\bfa=(a_1,a_2,a_3,a_4,a_5)\in\mfp^{(5)}$,
define $$M(\bfa)=\{i\in\{1,2,3,4,5\} \mid |a_i|=\|\bfa\|\}.$$ That is,
$M(\bfa)$ comprises those coordinates where $\|\bfa\| =
\max\{|a_i|\mid i\in\{1,\dots,5\}\}$ is attained. Define the following
mutually disjoint subsets of $\mfp^{(5)}$:
\begin{align*}
\Omega_{1} &=\{\bfa\in\mfp^{(5)}\mid 1\in M(\bfa)\},\\
\Omega_{2} &=\{\bfa\in\mfp^{(5)}\mid 1\notin M(\bfa),\, 
M(\bfa)\cap\{2,3\}=\{2\}\},\\
\Omega_{3} &=\{\bfa\in\mfp^{(5)}\mid 1\notin M(\bfa),\, 
M(\bfa)\cap\{2,3\}=\{3\}\},\\
\Omega_{4} &=\{\bfa\in\mfp^{(5)}\mid 1\notin M(\bfa),\, 
M(\bfa)\cap\{2,3\}=\{2,3\}\},\\
\Omega_{5} &=\{\bfa\in\mfp^{(5)}\mid 1\notin M(\bfa),\, 
M(\bfa)\cap\{2,3\}=\varnothing,\, 
M(\bfa)\cap\{4,5\}=\{4\}\},\\
\Omega_{6} &=\{\bfa\in\mfp^{(5)}\mid 1\notin M(\bfa),\, 
M(\bfa)\cap\{2,3\}=\varnothing,\, 
M(\bfa)\cap\{4,5\}=\{5\}\},\\
\Omega_{7} &=\{\bfa\in\mfp^{(5)}\mid 1\notin M(\bfa),\, 
M(\bfa)\cap\{2,3\}=\varnothing,\, 
M(\bfa)\cap\{4,5\}=\{4,5\}\}.
\end{align*}
\iffalse
\begin{align*}
\Omega_{1} &=\{\bfa\in\mfp^{(5)}\mid x_1\in M(\bfa)\},\\
\Omega_{2} &=\{\bfa\in\mfp^{(5)}\mid x_1\notin M(\bfa),\, 
M(\bfa)\cap\{x_2,x_3\}=\{x_2\}\},\\
\Omega_{3} &=\{\bfa\in\mfp^{(5)}\mid x_1\notin M(\bfa),\, 
M(\bfa)\cap\{x_2,x_3\}=\{x_3\}\},\\
\Omega_{4} &=\{\bfa\in\mfp^{(5)}\mid x_1\notin M(\bfa),\, 
M(\bfa)\cap\{x_2,x_3\}=\{x_2,x_3\}\},\\
\Omega_{5} &=\{\bfa\in\mfp^{(5)}\mid x_1\notin M(\bfa),\, 
M(\bfa)\cap\{x_2,x_3\}=\varnothing,\, 
M(\bfa)\cap\{x_4,x_5\}=\{x_4\}\},\\
\Omega_{6} &=\{\bfa\in\mfp^{(5)}\mid x_1\notin M(\bfa),\, 
M(\bfa)\cap\{x_2,x_3\}=\varnothing,\, 
M(\bfa)\cap\{x_4,x_5\}=\{x_5\}\},\\
\Omega_{7} &=\{\bfa\in\mfp^{(5)}\mid x_1\notin M(\bfa),\, 
M(\bfa)\cap\{x_2,x_3\}=\varnothing,\, 
M(\bfa)\cap\{x_4,x_5\}=\{x_4,x_5\}\}.
\end{align*}
\fi
We obtain a partition $\mfp^{(5)}=\bigdotcup_{i=1}^{7}\Omega_{i}$ and
a corresponding decomposition
\[
I_{D_{2}}^{3}=q^{-5}\sum_{i=1}^{7}I_{D_{2}}^{3,i},
\]
where, for $i=1,\dots,7$, we set
\[
I_{D_{2}}^{3,i}:=\int_{\Omega_{i}}|x|^{\tau}\|x,y,\tilde{y}_{1},\tilde{y}_{2},\tilde{y}_{3}\|^{2\rho}\frac{\|xy,x\tilde{y}_{1},x\tilde{y}_{2},x\tilde{y}_{3},y\tilde{y}_{1}\|^{2\rho}}{\|y,\tilde{y}_{1},\tilde{y}_{2},\tilde{y}_{3}\|^{2\rho}}d\mu(x,y,\tilde{\bfy}).
\]
We compute each of the integrals $I_{D_{2}}^{3,i}$
in turn.

\subsubsection{Subcase 1}

Assume that $(x,y,\tilde{y}_{1},\tilde{y}_{2},\tilde{y}_{3})\in\Omega_{1}$.
Then $|x|\geq\|y,\tilde{y}_{1},\tilde{y}_{2},\tilde{y}_{3}\|$, so
$|x\tilde{y}_{1}|\geq|y\tilde{y}_{1}|$, and 
\[
\|xy,x\tilde{y}_{1},x\tilde{y}_{2},x\tilde{y}_{3},y\tilde{y}_{1}\|=\|xy,x\tilde{y}_{1},x\tilde{y}_{2},x\tilde{y}_{3}\|=|x|\cdot\|y,\tilde{y}_{1},\tilde{y}_{2},\tilde{y}_{3}\|.
\]
Thus
\[
I_{D_{2}}^{3,1}=\int_{\Omega_{1}}|x|^{\tau+4\rho}d\mu(x,y,\tilde{\bfy}).
\]
For $X\in\N$,
\begin{align*}
\lefteqn{\mu(\{(x,y,\tilde{y}_{1},\tilde{y}_{2},\tilde{y}_{3}) \in\Omega_{1}\mid v(x)=X\})}\\
 &=\mu(\{(x,y,\tilde{y}_{1},\tilde{y}_{2},\tilde{y}_{3})\in\mfp^{(5)}\mid v(x)=X,\, X\leq\min\{v(y),v(\tilde{y}_{1}),v(\tilde{y}_{2}),v(\tilde{y}_{3})\})\\
&=\mu(\{(x,y,\tilde{y}_{1},\tilde{y}_{2},\tilde{y}_{3})\in\mfp^{X}\setminus\mfp^
{X+1}\times(\mfp^{X})^{4}\})=(1-q^{-1})q^{-5X}.
\end{align*}
Thus
\begin{equation*}
I_{D_{2}}^{3,1} =\sum_{\substack{X\in\N}
}(1-q^{-1})q^{-5X}q^{-(\tau+4\rho)X} %= (1-q^{-1})\sum_{\substack{X\in\N}}q^{-(5+\tau+4\rho)X}\\
  =\boxed{\frac{q^{-5-\tau-4\rho}(1-q^{-1})}{1-q^{-5-\tau-4\rho}}}.
\end{equation*}
Assume now that $(x,y,\tilde{y}_{1},\tilde{y}_{2},\tilde{y}_{3})\notin\Omega_{1}$,
that is, $x\notin M$. Then $|x|<\|y,\tilde{y}_{1},\tilde{y}_{2},\tilde{y}_{3}\|$,
and so 
\begin{equation}
I_{D_{2}}^{3,i}=\int_{\Omega_{i}}|x|^{\tau}\|xy,x\tilde{y}_{1},x\tilde{y}_{2},x\tilde{y}_{3},y\tilde{y}_{1}\|^{2\rho}d\mu(x,\tilde{\bfy}),\nonumber%\label{eq:Int-Case3-II}
\end{equation}
for all $2\leq i\leq7$. We treat these six remaining cases in what
follows.

\subsubsection{\label{sub:Case-2}Subcase 2}

We have
\[
I_{D_{2}}^{3,2}=\int_{\Omega_{2}}|x|^{\tau}|y|^{2\rho}\|x,\tilde{y}_{1}\|^{2\rho}d\mu(x,y,\tilde{\bfy})
\]
and, for fixed $(X,Y,\tilde{Y}_{1},\tilde{Y}_{2},\tilde{Y}_{3})\in\N^{5}$,
\begin{multline*}
\mu(\{(x,y,\tilde{y}_{1},\tilde{y}_{2},\tilde{y}_{3}) \in\Omega_{2}\mid v(x)=X,v(y)=Y,v(\tilde{y}_{i})=\tilde{Y}_{i},\text{ for }i=1,2,3\})\\
  =\begin{cases}
(1-q^{-1})^{5}q^{-X-Y-\tilde{Y}_{1}-\tilde{Y}_{2}-\tilde{Y}_{3}} & \text{if }Y<X,\;Y<\tilde{Y}_{1},\;Y\leq\tilde{Y}_{2},\;Y\leq\tilde{Y}_{3},\\
0 & \text{otherwise}.
\end{cases}
\end{multline*}
We thus get
\begin{align*}
I_{D_{2}}^{3,2} &=(1-q^{-1})^{5}\sum_{\substack{(X,Y,\tilde{Y}_{i})\in\N^{5}\\
Y<X,\;Y<\tilde{Y}_{1},\;Y\leq\tilde{Y}_{2},\;Y\leq\tilde{Y}_{3}
}
}q^{-X-Y-\tilde{Y}_{1}-\tilde{Y}_{2}-\tilde{Y}_{3}}q^{-\tau X}q^{-2\rho Y}q^{-2\rho\min\{X,\tilde{Y}_{1}\}}\\
 &=(1-q^{-1})^{5}\sum_{\substack{(X,Y,\tilde{Y}_{i})\in\N^{5}\\
Y<X,\;Y<\tilde{Y}_{1},\;Y\leq\tilde{Y}_{2},\;Y\leq\tilde{Y}_{3}
}
}q^{-(1+\tau)X}q^{-(1+2\rho)Y}q^{-\tilde{Y}_{1}}q^{-\tilde{Y}_{2}}q^{-\tilde{Y}_{3}}q^{-2\rho\min\{X,\tilde{Y}_{1}\}}.
\end{align*}
Set $X=Y+X'$, $\tilde{Y}_{1}=Y+\tilde{Y}_{1}'$,
$\tilde{Y}_{2}=Y+\tilde{Y}_{2}'$, $\tilde{Y}_{3}=Y+\tilde{Y}_{3}'$,
where $X',\tilde{Y}_{1}'\in\N$ and
$\tilde{Y}_{2}',\tilde{Y}_{3}'\in\N_{0}$. Using
Lemma~\ref{lem:Identities}\,\eqref{enu:AKOV-Identity}, we then get
\begin{align*}
I_{D_{2}}^{3,2} &= (1-q^{-1})^{5}\sum_{\substack{(X',Y,\tilde{Y}_{1}',\tilde{Y}_{2}',\tilde{Y}_{3}')\in\N^{3}\times\N_{0}^{2}}
}q^{-(1+\tau)(Y+X')}q^{-(1+2\rho)Y}q^{-(Y+\tilde{Y}_{1}')}q^{-(Y+\tilde{Y}_{2}')}q^{-(Y+\tilde{Y}_{3}')}\\
 & \quad\quad \cdot q^{-2\rho\min\{Y+X',Y+\tilde{Y}_{1}'\}}\\
 &= (1-q^{-1})^{5}\sum_{\substack{(X',Y,\tilde{Y}_{1}',\tilde{Y}_{2}',\tilde{Y}_{3}')\in\N^{3}\times\N_{0}^{2}}
}q^{-(1+\tau)X'}q^{-(5+\tau+4\rho)Y}q^{-\tilde{Y}_{1}'}q^{-\tilde{Y}_{2}'}q^{-\tilde{Y}_{3}'}q^{-2\rho\min\{X',\tilde{Y}_{1}'\}}\\
 &= (1-q^{-1})^{5}\frac{1}{(1-q^{-1})^{2}}\frac{q^{-5-\tau-4\rho}}{1-q^{-5-\tau-4\rho}}\sum_{\substack{(X',\tilde{Y}_{1}')\in\N^{2}}
}q^{-(1+\tau)X'}q^{-\tilde{Y}_{1}'}q^{-2\rho\min\{X',\tilde{Y}_{1}'\}}\\
 &= (1-q^{-1})^{5}\frac{1}{(1-q^{-1})^{2}}\frac{q^{-5-\tau-4\rho}}{1-q^{-5-\tau-4\rho}}\frac{q^{-2-\tau-2\rho}(1-q^{-2-\tau})}{(1-q^{-2-\tau-2\rho})(1-q^{-1-\tau})(1-q^{-1})}\\
 &= \boxed{\frac{q^{-7-2\tau-6\rho}(1-q^{-1})^{2}(1-q^{-2-\tau})}{(1-q^{-5-\tau-4\rho})(1-q^{-2-\tau-2\rho})(1-q^{-1-\tau})}}.
\end{align*}

\subsubsection{Subcase 3}

We have
\[
I_{D_{2}}^{3,3}=\int_{\Omega_{3}}|x|^{\tau}|\tilde{y}_{1}|^{2\rho}\|x,y\|^{2\rho}d\mu(x,y,\tilde{\bfy}).
\]
Since this integral is given by that in Section~\ref{sub:Case-2} by
permuting the variables $y$ and $\tilde{y}_{1}$, the two integrals
coincide, that is, $I_{D_{2}}^{3,3}=I_{D_{2}}^{3,2}$.
%\[
%I_{D_{2}}^{3,3}=I_{D_{2}}^{3,2}=\boxed{\frac{q^{-7-2\tau-6\rho}(1-q^{-1})^{2}(1-q^{-2-\tau})}{(1-q^{-5-\tau-4\rho})(1-q^{-2-\tau-2\rho})(1-q^{-1-\tau})}}.
%\]

\subsubsection{Subcase 4}

We have
\[
I_{D_{2}}^{3,4}=\int_{\Omega_{4}}|x|^{\tau}|\tilde{y}_{1}|^{2\rho}\|x,y\|^{2\rho}d\mu(x,y,\tilde{\bfy})
\]
and, for fixed $(X,Y,\tilde{Y}_{1},\tilde{Y}_{2},\tilde{Y}_{3})\in\N^{5}$,
\begin{multline*}
\mu(\{(x,y,\tilde{y}_{1},\tilde{y}_{2},\tilde{y}_{3})  \in\Omega_{4}\mid v(x)=X,v(y)=Y,v(\tilde{y}_{i})=\tilde{Y}_{i},\text{ for }i=1,2,3\})\\
= \begin{cases}
(1-q^{-1})^{5}q^{-X-Y-\tilde{Y}_{1}-\tilde{Y}_{2}-\tilde{Y}_{3}} & \text{if }Y<X,\;Y=\tilde{Y}_{1},\;Y\leq\tilde{Y}_{2},\;Y\leq\tilde{Y}_{3},\\
0 & \text{otherwise}.
\end{cases}
\end{multline*}
We thus get
\begin{align*}
I_{D_{2}}^{3,4} &=(1-q^{-1})^{5}\sum_{\substack{(X,Y,\tilde{Y}_{i})\in\N^{5}\\
Y<X,\;Y=\tilde{Y}_{1},\;Y\leq\tilde{Y}_{2},\;Y\leq\tilde{Y}_{3}
}
}q^{-X-Y-\tilde{Y}_{1}-\tilde{Y}_{2}-\tilde{Y}_{3}}q^{-\tau X}q^{-2\rho\tilde{Y}_{1}}q^{-2\rho\min\{X,Y\}}\\
 &=(1-q^{-1})^{5}\sum_{\substack{(X,Y,\tilde{Y}_{2},\tilde{Y}_{3})\in\N^{4}\\
Y<X,\;Y\leq\tilde{Y}_{2},\;Y\leq\tilde{Y}_{3}
}
}q^{-(1+\tau)X}q^{-2(1+\rho)Y}q^{-\tilde{Y}_{2}}q^{-\tilde{Y}_{3}}q^{-2\rho\min\{X,Y\}}.
\end{align*}
Set $X=Y+X'$, $\tilde{Y}_{2}=Y+\tilde{Y}_{2}'$,
$\tilde{Y}_{3}=Y+\tilde{Y}_{3}'$, where $X'\in\N$, $\tilde{Y}_{2}',\tilde{Y}_{3}'\in\N_{0}$. We then get
\begin{align*}
\lefteqn{I_{D_{2}}^{3,4}}\\ &=(1-q^{-1})^{5}\sum_{(X',Y,\tilde{Y}_{2},\tilde{Y}_{3})\in\N^{2}\times\N_{0}^{2}}q^{-(1+\tau)(Y+X')}q^{-2(1+\rho)Y}q^{-(Y+\tilde{Y}_{2})}q^{-(Y+\tilde{Y}_{3})}\\ &\quad\quad \cdot q^{-2\rho\min\{Y+X',Y\}}\\
 &=(1-q^{-1})^{5}\sum_{(X',Y,\tilde{Y}_{2},\tilde{Y}_{3})\in\N^{2}\times\N_{0}^{2}}q^{-(1+\tau)X'}q^{-(5+\tau+4\rho)Y}q^{-\tilde{Y}_{2}}q^{-\tilde{Y}_{3}}\\
 &=(1-q^{-1})^{5}\frac{q^{-1-\tau}}{1-q^{-1-\tau}}\frac{q^{-5-\tau-4\rho}}{1-q^{-5-\tau-4\rho}}\frac{1}{(1-q^{-1})^{2}} = \boxed{\frac{q^{-6-2\tau-4\rho}(1-q^{-1})^{3}}{(1-q^{-1-\tau})(1-q^{-5-\tau-4\rho})}}.
\end{align*}

\subsubsection{\label{sub:Case-5}Subcase 5}

We have
\[
I_{D_{2}}^{3,5}=\int_{\Omega_{5}}|x|^{\tau}\|x\tilde{y}_{2},y\tilde{y}_{1}\|^{2\rho}d\mu(x,y,\tilde{\bfy})
\]
and, for fixed $(X,Y,\tilde{Y}_{1},\tilde{Y}_{2},\tilde{Y}_{3})\in\N^{5}$,
\begin{multline*}
\mu(\{(x,y,\tilde{y}_{1},\tilde{y}_{2},\tilde{y}_{3}) \in\Omega_{5}\mid v(x)=X,v(y)=Y,v(\tilde{y}_{i})=\tilde{Y}_{i},\text{ for }i=1,2,3\})\\
  =\begin{cases}
(1-q^{-1})^{5}q^{-X-Y-\tilde{Y}_{1}-\tilde{Y}_{2}-\tilde{Y}_{3}} & \text{if }\tilde{Y}_{2}<X,\;\tilde{Y}_{2}<Y,\;\tilde{Y}_{2}<\tilde{Y}_{1},\;\tilde{Y}_{2}<\tilde{Y}_{3},\\
0 & \text{otherwise}.
\end{cases}
\end{multline*}
We thus get
\begin{align*}
  I_{D_{2}}^{3,5} &=(1-q^{-1})^{5}\sum_{\substack{(X,Y,\tilde{Y}_{i})\in\N^{5}\\
      \tilde{Y}_{2}<X,\;\tilde{Y}_{2}<Y,\;\tilde{Y}_{2}<\tilde{Y}_{1},\;\tilde{Y}_{2}<\tilde{Y}_{3}
    }
  }q^{-X-Y-\tilde{Y}_{1}-\tilde{Y}_{2}-\tilde{Y}_{3}}q^{-\tau X}q^{-2\rho\min\{X+\tilde{Y}_{2},Y+\tilde{Y}_{1}\}}\\
  &= (1-q^{-1})^{5}\sum_{\substack{(X,Y,\tilde{Y}_{i})\in\N^{5}\\
      \tilde{Y}_{2}<X,\;\tilde{Y}_{2}<Y,\;\tilde{Y}_{2}<\tilde{Y}_{1},\;\tilde{Y}_{2}<\tilde{Y}_{3}
    }
  }q^{-(1+\tau)X}q^{-Y}q^{-\tilde{Y}_{1}}q^{-\tilde{Y}_{2}}q^{-\tilde{Y}_{3}}q^{-2\rho\min\{X+\tilde{Y}_{2},Y+\tilde{Y}_{1}\}}.
\end{align*}
Set $X=\tilde{Y}_{2}+X'$, $Y=\tilde{Y}_{2}+Y'$, $\tilde{Y}_{1}=\tilde{Y}_{2}+\tilde{Y}_{1}'$,
$\tilde{Y}_{3}=\tilde{Y}_{2}+\tilde{Y}_{3}'$, for $X',Y',\tilde{Y}_{1}',\tilde{Y}_{3}'\in\N$.
Then
\begin{align*}
\lefteqn{I_{D_{2}}^{3,5} }\\&=(1-q^{-1})^{5}\sum_{(X',Y',\tilde{Y}_{1}',\tilde{Y}_{2},\tilde{Y}_{3}')\in\N^{5}}q^{-(1+\tau)(\tilde{Y}_{2}+X')}q^{-(\tilde{Y}_{2}+Y')}q^{-(\tilde{Y}_{2}+\tilde{Y}_{1}')}q^{-\tilde{Y}_{2}}q^{-(\tilde{Y}_{2}+\tilde{Y}_{3}')}\\
 &\quad\quad \cdot q^{-2\rho\min\{\tilde{Y}_{2}+X'+\tilde{Y}_{2},\tilde{Y}_{2}+Y'+\tilde{Y}_{2}+\tilde{Y}_{1}'\}}\\
 &=(1-q^{-1})^{5}\sum_{(X',Y',\tilde{Y}_{1}',\tilde{Y}_{2},\tilde{Y}_{3}')\in\N^{5}}q^{-(5+\tau+4\rho)\tilde{Y}_{2}}q^{-(1+\tau)X'}q^{-Y'}q^{-\tilde{Y}_{1}'}q^{-\tilde{Y}_{3}'}q^{-2\rho\min\{X',Y'+\tilde{Y}_{1}'\}}\\
 &=(1-q^{-1})^{5}\frac{q^{-5-\tau-4\rho}}{1-q^{-5-\tau-4\rho}}\frac{q^{-1}}{1-q^{-1}}\sum_{(X',Y',\tilde{Y}_{1}')\in\N^{3}}q^{-(1+\tau)X'}q^{-(Y'+\tilde{Y}_{1}')}q^{-2\rho\min\{X',Y'+\tilde{Y}_{1}'\}}.
\end{align*}
Applying Lemma~\ref{lem:Identities}\,\eqref{enu:Identity-2} we get
\begin{align*}
\lefteqn{{I_{D_{2}}^{3,5}}} \\&=(1-q^{-1})^{4}\frac{q^{-6-\tau-4\rho}}{1-q^{-5-\tau-4\rho}}\frac{q^{-3-\tau-2\rho}(1-q^{-1-\tau}+q^{-1-\tau-2\rho}-2q^{-2-\tau-2\rho}+q^{-4-2\tau-2\rho})}{(1-q^{-2-\tau-2\rho})^{2}(1-q^{-1-\tau})(1-q^{-1})^{2}}\\
 &=\boxed{\frac{q^{-9-2\tau-6\rho}(1-q^{-1})^{2}(1-q^{-1-\tau}+q^{-1-\tau-2\rho}-2q^{-2-\tau-2\rho}+q^{-4-2\tau-2\rho})}{(1-q^{-5-\tau-4\rho})(1-q^{-2-\tau-2\rho})^{2}(1-q^{-1-\tau})}}.
\end{align*}

\subsubsection{Subcase 6}

We have
\[
I_{D_{2}}^{3,6}=\int_{\Omega_{6}}|x|^{\tau}\|x\tilde{y}_{3},y\tilde{y}_{1}\|^{2\rho}d\mu(x,y,\tilde{\bfy}).
\]
Since this integral is given by that in Section~\ref{sub:Case-5} by
permuting the variables $\tilde{y}_{2}$ and $\tilde{y}_{3}$, the two
integrals coincide, that is, 
$I_{D_{2}}^{3,6}=I_{D_{2}}^{3,5}$.
%\begin{multline*}
% I_{D_{2}}^{3,6}=I_{D_{2}}^{3,5}=\boxed{\frac{q^{-9-2\tau-6\rho}(1-q^{-1})^{2}(1-q^{-1-\tau}+q^{-1-\tau-2\rho}-2q^{-2-\tau-2\rho}+q^{-4-2\tau-2\rho})}{(1-q^{-5-\tau-4\rho})(1-q^{-2-\tau-2\rho})^{2}(1-q^{-1-\t%au})}}.
%\end{multline*}

\subsubsection{Subcase 7}

We have
\[
I_{D_{2}}^{3,7}=\int_{\Omega_{7}}|x|^{\tau}\|x\tilde{y}_{2},y\tilde{y}_{1}\|^{2\rho}d\mu(x,y,\tilde{\bfy}).
\]
and, for fixed $(X,Y,\tilde{Y}_{1},\tilde{Y}_{2},\tilde{Y}_{3})\in\N^{5}$,
\begin{multline*}
\mu(\{(x,y,\tilde{y}_{1},\tilde{y}_{2},\tilde{y}_{3}) \in\Omega_{7}\mid v(x)=X,v(y)=Y,v(\tilde{y}_{i})=\tilde{Y}_{i},\text{ for }i=1,2,3\})\\
 =\begin{cases}
(1-q^{-1})^{5}q^{-X-Y-\tilde{Y}_{1}-\tilde{Y}_{2}-\tilde{Y}_{3}} & \text{if }\tilde{Y}_{2}<X,\;\tilde{Y}_{2}<Y,\;\tilde{Y}_{2}<\tilde{Y}_{1},\;\tilde{Y}_{2}=\tilde{Y}_{3},\\
0 & \text{otherwise}.
\end{cases}
\end{multline*}
We thus get
\begin{align*}
I_{D_{2}}^{3,7} &=(1-q^{-1})^{5}\sum_{\substack{(X,Y,\tilde{Y}_{i})\in\N^{5}\\
\tilde{Y}_{2}<X,\;\tilde{Y}_{2}<Y,\;\tilde{Y}_{2}<\tilde{Y}_{1},\;\tilde{Y}_{2}=\tilde{Y}_{3}
}}q^{-X-Y-\tilde{Y}_{1}-\tilde{Y}_{2}-\tilde{Y}_{3}}q^{-\tau X}q^{-2\rho\min\{X+\tilde{Y}_{2},Y+\tilde{Y}_{1}\}}\\
 &=(1-q^{-1})^{5}\sum_{\substack{(X,Y,\tilde{Y}_{i})\in\N^{5}\\
\tilde{Y}_{2}<X,\;\tilde{Y}_{2}<Y,\;\tilde{Y}_{2}<\tilde{Y}_{1},\;\tilde{Y}_{2}=\tilde{Y}_{3}
}
}q^{-(1+\tau)X}q^{-Y}q^{-\tilde{Y}_{1}}q^{-2\tilde{Y}_{2}}q^{-2\rho\min\{X+\tilde{Y}_{2},Y+\tilde{Y}_{1}\}}.
\end{align*}
Set $X=\tilde{Y}_{2}+X'$, $Y=\tilde{Y}_{2}+Y'$, $\tilde{Y}_{1}=\tilde{Y}_{2}+\tilde{Y}_{1}'$,
for $X',Y',\tilde{Y}_{1}'\in\N$. We then get
\begin{align*}
\lefteqn{(1-q^{-1})^{-5}I_{D_{2}}^{3,7}}\\&=\sum_{(X',Y',\tilde{Y}_{1}',\tilde{Y}_{2})\in\N^{4}}q^{-(1+\tau)(\tilde{Y}_{2}+X')}q^{-(\tilde{Y}_{2}+Y')}q^{-(\tilde{Y}_{2}+\tilde{Y}_{1}')}q^{-2\tilde{Y}_{2}}q^{-2\rho\min\{\tilde{Y}_{2}+X'+\tilde{Y}_{2},\tilde{Y}_{2}+Y'+\tilde{Y}_{2}+\tilde{Y}_{1}'\}}\\
 &=\sum_{(X',Y',\tilde{Y}_{1}',\tilde{Y}_{2})\in\N^{4}}q^{-(5+\tau+4\rho)\tilde{Y}_{2}}q^{-(1+\tau)X'}q^{-Y'}q^{-\tilde{Y}_{1}'}q^{-2\rho\min\{X',Y'+\tilde{Y}_{1}'\}}\\
 &=\frac{q^{-5-\tau-4\rho}}{1-q^{-5-\tau-4\rho}}\sum_{(X',Y',\tilde{Y}_{1}')\in\N^{3}}q^{-(1+\tau)X'}q^{-(Y'+\tilde{Y}_{1}')}q^{-2\rho\min\{X',Y'+\tilde{Y}_{1}'\}}.
\end{align*}
Applying Lemma~\ref{lem:Identities}\,\eqref{enu:Identity-2} we get
$I_{D_{2}}^{3,7} = \frac{(1-q^{-1})}{q^{-1}}I_{D_{2}}^{3,5}$.

\subsection{Putting the pieces together}

The computations in Section~\ref{subsec:ID2.m=2} combine to an
explicit formula for
$$I_{D_2} = | \mathbb{A}^8(\F_q) \setminus
V_{3}(\F_{q})|I_{D_{2}}^{1}+(|V_{3}(\F_{q})|-|V_{2}(\F_{q})|)
I_{D_{2}}^{2} + |V_{2}(\F_{q})\setminus\{0\}| I_{D_{2}}^{3},$$
with $I_{D_2}^3 = p^{-5}\sum_{i=1}^7 I_{D_2}^{3,i}$.
Using Lemma~\ref{lem:V2V3-Fq} and equations
\eqref{equ:partition.Z.m=2} and \eqref{equ:ID1.solved}, we obtain,
with $a=q^{-1}$, $b=q^{-\tau}$, $c=q^{-\rho}$, that
\begin{align*}
\lefteqn{\calZ_{\lri}(\rho,\tau)=
\frac{ab(a-1)^2}{(1-a^5bc^4)(1-a^5bc^2)(1-a^2bc^2)^2(1-ab)}}\\
&\quad \Big(a^{22}b^4c^{10}+a^{21}b^4
  c^{10}+a^{20}b^4c^{10}+a^{19}b^4c^{10}+a^{18}b^4c^{10}\\
&\quad +a^{17}b^4c^{10}+a^{16}b^4c^{10}+a^{15}b^4c^{10}+a^{14}
  b^4c^{10}-2a^{20}b^3c^{10}-a^{19}b^3c^{10}-a^{18}b^3c^{10}\\
&\quad-2a^{17}b^3c^{10}-a^{15}b^3c^{10}-a^{14}b^3c^{
 10}-a^{13}b^3c^{10}-a^{19}b^3c^8-a^{18}b^3c^8-2a^{17}b^3c^8\\
&\quad-5a^{16}b^3c^8-4a^{15}b^3c^8-4a^{14}b^3c^8-3a^{13}b^3c^8-3a^{12}b^3c^8-a^{11}
  b^3c^8-a^{10}b^3c^8\\
&\quad-a^9b^3c^8+a^{18}b^2c^8+a^{17}b^2c^8+2a^{16}
  b^2c^8+5a^{15}b^2c^8+4a^{14}b^2c^8+4a^{13}b^2c^8\\
&\quad +3a^{12}b^2c^8+3a^{11}b^2c^8+a^{10}b^2c^8+a^9b^2
  c^8+a^8b^2c^8+a^{16}b^3c^6+a^{15}b^3c^6+a^{14}b^3c^6\\
&\quad -a^{12}b^3c^6-a^{11}b^3c^6-a^{10}b^3c^6-a^9b^3c^6-a
  ^{16}b^2c^6+a^{14}b^2c^6+2a^{13}b^2c^6+5a^{12}b^2c^6\\
&\quad+4a^{11}b^2c^6+6a^{10}b^2c^6+4a^9b^2c^6+3a^
  8b^2c^6+2a^7b^2c^6-2a^{13}bc^6-3a^{12}bc^6\\
&\quad-4a^{11}bc^6-4a^{10}bc^6-5a^9bc^6-3a^8bc^6-2a^
  7bc^6-2a^6bc^6-a^{13}b^2c^4-2a^{12}b^2c^4\\
&\quad-a^{11}b^2c^4-a^{10}b^2c^4+a^8b^2c^4+3a^7b^2c^4+a^6b^2
  c^4+a^5b^2c^4+a^4b^2c^4+a^{12}bc^4+2a^{11}bc^4\\
&\quad+a^9bc^4-2a^8bc^4-3a^7bc^4-4a^6bc^4-3a^5bc
  ^4-a^4bc^4-a^3bc^4+a^8c^4+2a^7c^4+2a^6c^4\\
&\quad+2a^5c^4+a^4c^4+a^8bc^2+a^7bc^2+2a^6bc^2-a^5bc^2-a^4
  bc^2-a^3bc^2-2a^2bc^2-a^7c^2\\
&\quad-a^6c^2-2a^5c^2+a^4c^2+a^3c^2+a^2c^2+2ac^2+a^5-a^4-a+1\Big)
\end{align*}

By \eqref{equ:zeta=integral.m2n1} and a direct computation, this
completes the proof of
Theorem~\ref{thm:m2n1}. 

Corollary~\ref{cor:m2n1} follows, for instance, along the lines of
\cite[Lemma~5.5]{duSWoodward/08}. Indeed, the invariant $\beta$
defined on \cite[p.~124]{duSWoodward/08} pertinent to the Euler
product
$$\prod_{\mfp}
(q^{8}t^{3}-q^{7}t^{2}+q^{6}t^{2}-q^{5}t^{2}-q^{3}t+q^{2}t-qt+1)$$
equals $\beta = \max\{8/3, 7/2,3/1\}= 7/2.$

\bibliographystyle{plain} \bibliography{masterbibliography21Feb15}

\end{document}